\documentclass[preprint,12pt,numbers,sort&compress]{elsarticle}
\usepackage{soul}
\usepackage{psfrag}
\usepackage{a4wide}
\usepackage{rotating}
\usepackage{color}
\usepackage{amssymb}
\usepackage{amsmath}
\usepackage{amsthm}
\usepackage{verbatim} 
\usepackage{tikz}
\usepackage{cancel}
\newcommand{\f}[1]{\boldsymbol{#1}}
\def\bfm#1{\boldsymbol{#1}}

\newcommand{\N}{\mathbb N}

\newcommand{\R}{\mathbb R}

\newtheorem{lem}{Lemma}

\theoremstyle{definition}
\newtheorem{ex}{Example}

\newtheorem{rem}{Remark}
\advance\textheight by 0.4cm
\advance\topmargin by -0.2cm

\definecolor{gold}{rgb}{1,0.7,0}
\definecolor{dred}{rgb}{0.92,0,0}
\definecolor{dgreen}{rgb}{0,0.6,0}

\def\new{\color{black}}
\definecolor{myred}{rgb}{1,0.2,0.2}
\definecolor{mygreen}{rgb}{0,0.5,0}
\definecolor{myblue}{rgb}{0.2,0.2,1}
\definecolor{mypurple}{rgb}{1,0.3,1}
\definecolor{myviolet}{rgb}{0.3,0,0.7}

\begin{document}

\begin{frontmatter}

\title{A locally based construction of analysis-suitable $G^1$ multi-patch spline surfaces}

\author[Ricam]{Andrea Farahat}
\ead{andrea.farahat@ricam.oeaw.ac.at}

\author[vil]{Mario Kapl}
\ead{m.kapl@fh-kaernten.at}

\author[slo1]{Alja\v z Kosma\v c}
\ead{aljaz.kosmac@iam.upr.si}

\author[slo1]{Vito Vitrih\corref{cor}}
\ead{vito.vitrih@upr.si}
 
\address[Ricam]{Flightkeys GmbH, Vienna, Austria; RICAM, Linz, Austria} 
 
\address[vil]{ADMiRE Research Center, Carinthia University of Applied Sciences, Villach, Austria}

\address[slo1]{IAM and FAMNIT, University of Primorska, Koper, Slovenia}

\cortext[cor]{Corresponding author}

\begin{abstract}
 Analysis-suitable~$G^1$~(AS-$G^1$) multi-patch spline surfaces~\cite{CoSaTa16} are particular $G^1$-smooth multi-patch spline surfaces, which are needed to ensure the construction of $C^1$-smooth multi-patch spline spaces with optimal polynomial reproduction properties~\cite{KaSaTa17b}. We present a novel local approach for the design of AS-$G^1$ multi-patch spline surfaces, which is based on the use of Lagrange multipliers. The presented method is simple and generates an AS-$G^1$ multi-patch spline surface by approximating a given $G^1$-smooth but non-AS-$G^1$ multi-patch surface. Several numerical examples demonstrate the potential of the proposed technique for the construction of AS-$G^1$ multi-patch spline surfaces and show that these surfaces are especially suited for applications in isogeometric analysis by solving the biharmonic problem, a particular fourth order partial differential equation, {\new with optimal rates of convergence} over them.
\end{abstract}


\end{frontmatter}

\section{Introduction}

Multi-patch spline surfaces with possibly extraordinary vertices, i.e.~vertices where less or more than four patches meet, are widely used in computer-aided design to represent complex surface domains~\cite{Fa97,HoLa93}. To employ these surfaces in the framework of isogeometric analysis~\cite{HuCoBa04} for solving fourth order partial differential equations such as the biharmonic problem, e.g.~\cite{BaDe15,FaJuKaTa22,NgKaPe15}, the Kirchhoff-Love shell problem, e.g.~\cite{FaVeKiKa23,kiendl-bletzinger-linhard-09,kiendl-bazilevs-hsu-wuechner-bletzinger-10}, the Cahn-Hilliard equation, e.g.~\cite{GoCaHu09,gomez2008isogeometric,KaMeBo16}, or problems of strain gradient elasticity, e.g.~\cite{gradientElast2011,MaReBeJu18,KhakaloNiiranenC1}, over them, the multi-patch spline surfaces have to be $G^1$-smooth~\cite{Pe02}, and preferably allow the construction of $C^1$-smooth isogeometric spline spaces {\new \cite{Pe15}} with optimal polynomial reproduction properties~\cite{CoSaTa16}. The resulting $C^1$-smooth isogeometric spline functions can then be used to solve the fourth order partial differential equations directly via their weak form and a standard Galerkin discretization, in general, with convergence rates of optimal order. 

As already mentioned, the employed multi-patch spline surface with possibly extraordinary vertices has to be $G^1$-smooth, which means that the surface possesses a well-defined tangent plane at all surface points~\cite{Pe02}, but which further means that the individual surface patches have to be regularly parameterized. The $G^1$-smoothness condition is a weaker condition than the $C^1$-smoothness condition, which even enforces that the two partial derivatives are well-defined in each point. A multi-patch spline surface with regular surface patches possessing extraordinary vertices cannot be $C^1$-smooth, and can become $C^1$-smooth just by adding singularities at the extraordinary vertices~\cite{Pe91}. The use of specific degenerate patches, e.g. D-patches~\cite{Re97}, for modeling such a surface allow the design of $C^1$-smooth isogeometric spline spaces with good approximation properties, see e.g. \cite{NgPe16,ToSpHu17}. However, one drawback is that the resulting spline spaces have to be considered with special care in the neighborhood of a singularity especially with respect to the numerical integration. Strongly related to the singular surface approach is the use of subdivision surfaces, see {\new e.g.~\cite{BHU10b,CiOrSch00,Peters2,RiAuFe16,ZhSaCi18}}, where again in the vicinity of an extraordinary vertex the numerical integration is challenging, and furthermore, the approximation properties are in general reduced.

There are two common techniques to model $G^1$-smooth multi-patch spline surfaces with extraordinary vertices {\new which are especially suited for the numerical simulation}.  
The first approach generates a multi-patch spline surface which is $C^1$-smooth in the regular regions of the surface and $G^1$-smooth in the irregular regions, i.e. in the neighborhood of extraordinary vertices, like a G-spline surface~\cite{Re95}. {\new Further possible examples are \cite{Pe15-2,NgKaPe15,KaPe17,KaPe18}.} 
The second approach constructs a multi-patch spline surface which is in general just $G^1$-smooth everywhere, see e.g.~{\new \cite{BlMoXu2020,FaJuKaTa22,MaMaMo2024,PaZoToGuCh23}}. 
$G^1$-smooth multi-patch spline surfaces which ensure now the construction of $C^1$-smooth isogeometric spline spaces with optimal polynomial reproduction properties are called analysis-suitable $G^1$, in short AS-$G^1$, and have been firstly introduced in~\cite{CoSaTa16}. The AS-$G^1$ condition is a more restrictive condition than the $G^1$ condition, and is fully determined by the parameterizations of the single surface patches. 
{\new The} class of {\new AS-$G^1$} multi-patch spline surfaces is of interest and importance, since the use of non-AS-$G^1$-smooth multi-patch spline surfaces can lead to $C^1$-smooth isogeometric spline spaces with dramatically reduced approximation properties~\cite{KaSaTa17b}. {\new More precisely, it was shown in \cite{KaSaTa17b} on the basis of several examples by performing $L^2$ approximation that the convergence for AS-$G^1$ multi-patch spline geometries is optimal with respect to the $L^2$ and $L^{\infty}$-norm, i.e. of order~$\mathcal{O}(h^{4})$ for bicubic $C^1$ splines, while the convergence for general non-AS-$G^1$ multi-patch spline geometries is dramatically reduced to an order~$\mathcal{O}(h^{\frac{1}{2}})$ for the $L^2$-norm and even to no convergence for the $L^{\infty}$-norm.} 

$G^1$-smooth multi-patch spline surfaces constructed by the first method are usually not AS-$G^1$ and lead to $C^1$-smooth isogeometric spline spaces with reduced approximation properties as shown {\new e.g.}~for the case of $G$-spline surfaces in \cite{WeFaLiWeCa23}. To get isogeometric spline spaces with optimal 
{\new polynomial reproduction properties}, it is necessary to partially increase the degree of the surface and of the isogeometric spline functions in the vicinity of the extraordinary vertices, see e.g.~\cite{Pe15-2,NgKaPe15,KaPe17,KaPe18}. The second approach in contrast allows the modeling of AS-$G^1$ multi-patch spline surfaces. Moreover, the resulting AS-$G^1$ multi-patch spline surfaces and their associated $C^1$-smooth isogeometric spline spaces still have the same degree everywhere, see e.g.~\cite{FaJuKaTa22,KaSaTa17b}.

AS-$G^1$ multi-patch spline geometries have been firstly studied for the case of planar (mapped) piecewise bilinear multi-patch domains~\cite{BeMa14,CoSaTa16,KaBuBeJu16,KaViJu15,mourrain2015geometrically}. Beyond (mapped) bilinear parameterizations, the construction of AS-$G^1$ multi-patch spline geometries is so far limited to the techniques~\cite{FaJuKaTa22,KaSaTa17b}. The method~\cite{KaSaTa17b} generates from a given planar non-AS-$G^1$ multi-patch parameterization an AS-$G^1$ multi-patch spline geometry by interpolating the boundary and by reparameterizing the interior of the domain. The proposed technique relies on the solving of a quadratic optimization problem with linear side constraints, and has been generalized also to a first simple surface example. Two drawbacks of the method are that the algorithm is fully global and that for the solving of the problem the use of symbolic computation is needed, which make the technique just applicable to multi-patch spline geometries with a relatively small number of patches. In~\cite{FaJuKaTa22}, a general framework for the construction of AS-$G^1$ multi-patch spline surface from given $G^1$-smooth but non-AS-$G^1$ multi-patch surface by means of $L^2$ projection is presented. The approach is based on the design of an AS-$G^1$ multi-patch template mapping and on the use of an associated $C^1$-smooth isogeometric spline space for approximating the given non-AS-$G^1$ multi-patch surface. However, the technique is again fully global, and in addition, the construction of the AS-$G^1$ multi-patch template mapping is so far available only for specific classes of multi-patch surfaces such as planarizable surfaces or surfaces topologically equivalent to the sphere. Having an AS-$G^1$ multi-patch spline surface, the usage of the $C^1$-smooth isogeometric spline space~\cite{KaSaTa19a} for the planar case and its extension~\cite{FaJuKaTa22} to the surface case allow to solve fourth order partial differential equations such as the biharmonic problem with optimal convergence rates over the surface. 

In this work, we introduce now a local method which approximates a given $G^1$-smooth but non-AS-$G^1$-smooth multi-patch surface by an AS-$G^1$ multi-patch spline surface. The developed local approach solves first for each inner vertex and boundary vertex of patch valency greater than one, then for each interface and finally for each surface patch of the AS-$G^1$ multi-patch spline surface a small quadratic optimization problem to determine the control points of the surface. Thereby, each minimization problem is solved numerically, which can be done in parallel in the three steps, i.e. for all vertices, for all interfaces and for all surface patches, and further involves Lagrange multipliers to enforce the AS-$G^1$ condition between two neighboring surface patches. Our proposed method is applicable to any $G^1$-smooth multi-patch surface, and can deal due to its local behaviour with multi-patch surfaces possessing a high number of patches. To demonstrate the potential of our method, several examples of AS-$G^1$ multi-patch spline geometries are constructed representing planar as well as surface domains. In addition, the $C^1$-smooth isogeometric spline space~\cite{FaJuKaTa22} is used to solve the biharmonic problem over some of the generated AS-$G^1$ multi-patch spline surfaces, where the numerical results show as required convergence rates of optimal order.     

The remainder of the paper is organized as follows. Section~\ref{sec:preliminaries} presents preliminaries including the description of the used multi-patch surface structure, the recalling of the class of AS-$G^1$ multi-patch surfaces as well as the introduction of the concept of $C^1$-smooth isogeometric spline functions over AS-$G^1$ multi-patch spline surfaces. In Section~\ref{sec:design_ASG1_surfaces}, we first specify the statement of the problem and describe necessary precomputations. Then, we present the locally based technique for the construction of AS-$G^1$ multi-patch spline surfaces. Section~\ref{sec:numerical_examples} demonstrates the potential of the novel method on the basis of several examples by generating AS-$G^1$ multi-patch spline surfaces from given $G^1$-smooth but non-AS-$G^1$ multi-patch surfaces, and by solving the biharmonic problem over some of them. Finally, we conclude the paper in Section~\ref{sec:conclusion}.  

\section{Preliminaries} \label{sec:preliminaries}

We will first present the used multi-patch surface setting, will then introduce the concept of AS-$G^1$ multi-patch spline surfaces which form a particular class of $G^1$-smooth multi-patch spline surfaces, and will finally recall the idea of $C^1$-smooth isogeometric spline spaces over the class of AS-$G^1$ multi-patch spline surfaces.   

\subsection{Multi-patch surface structure} \label{subsec:multi-patch_setting}

Let $p \geq 1$, $0 \leq r \leq p$ and $k \geq 0$. We denote by $\mathcal{S}^{p,r}_{k}$ the univariate spline space of degree~$p$ and regularity~$r$ with $k$ uniform inner knots of multiplicity~$p-r$ over the unit interval~$[0,1]$, and by $\mathcal{S}^{\f{p},\f{r}}_{\f{k}}$ with $\f{p}=(p,p)$, $\f{r}=(r,r)$ and $\f{k}=(k,k)$, the associated tensor-product spline space~$\mathcal{S}_{k}^{p,r} \times \mathcal{S}_{k}^{p,r}$ over the unit square~$[0,1]^2$. Let $N_{j}^{p,r}$, $j \in \mathcal{J} = \{0,1,\ldots,n-1 \}$, with $n=\dim \mathcal{S}_{k}^{p,r} = p+1+k(p-r)$, be the B-splines of the spline space~$\mathcal{S}_{k}^{p,r}$, and let $N_{j_1,j_2}^{\f{p},\f{r}} = N_{j_1}^{p,r} N_{j_2}^{p,r}$, $j_1,j_2 \in \mathcal{J}$, be the tensor-product B-splines of the spline space~$\mathcal{S}_{\f{k}}^{\f{p},\f{r}}$. In addition, let us denote by $\zeta_{j}^{p,r}$, $j \in \mathcal{J}$, the Greville abscissae of the spline space $\mathcal{S}^{p,r}_{k}$. We further need the spline spaces~$\mathcal{S}^{p,r+1}_{k}$ and $\mathcal{S}^{p-1,r}_{k}$ with the corresponding B-splines $N_{j}^{p,r+1}$, $j \in \mathcal{J}_0=\{0,1,\ldots,n_0-1 \}$, and $N_{j}^{p,r-1}$, $j \in \mathcal{J}_1=\{0,1,\ldots,n_1-1 \}$, respectively, with $n_0=\dim \mathcal{S}^{p,r+1}_k = p+1+k(p-r-1)$ and $n_1=\dim \mathcal{S}^{p-1,r}_k = p+k(p-r-1)$.

We consider a $G^1$-smooth conforming multi-patch spline surface~$\f{F}$, consisting of regular quadrilateral surface patch parameterizations $\f{F}^{(i)} \in (\mathcal{S}_{\f{k}}^{\f{p},\f{r}})^3$, $i \in \mathcal{I}_{\Omega}$, with B-spline representations
\begin{equation} \label{eq:Bspline_representation}
 \f{F}^{(i)}(\f{\xi})=\f{F}^{(i)}(\xi_1,\xi_2) = \sum_{j_1 \in \mathcal{J}} \sum_{j_2 \in \mathcal{J}} \f{d}_{j_1,j_2}^{(i)} N_{j_1,j_2}^{\f{p},\f{r}}(\xi_1,\xi_2), \quad \f{d}_{j_1,j_2}^{(i)} \in \R^3,
\end{equation}
which defines a surface domain~$\Omega$ via $\Omega= \cup_{i \in \mathcal{I}_{\Omega}} \overline{\Omega^{(i)}}$ with $\overline{\Omega^{(i)}} = \f{F}^{(i)}([0,1]^2)$, {\new cf.~Fig.~\ref{Domain01}.} 
\begin{figure}[htb]
	\centering
		\includegraphics[scale=0.27]{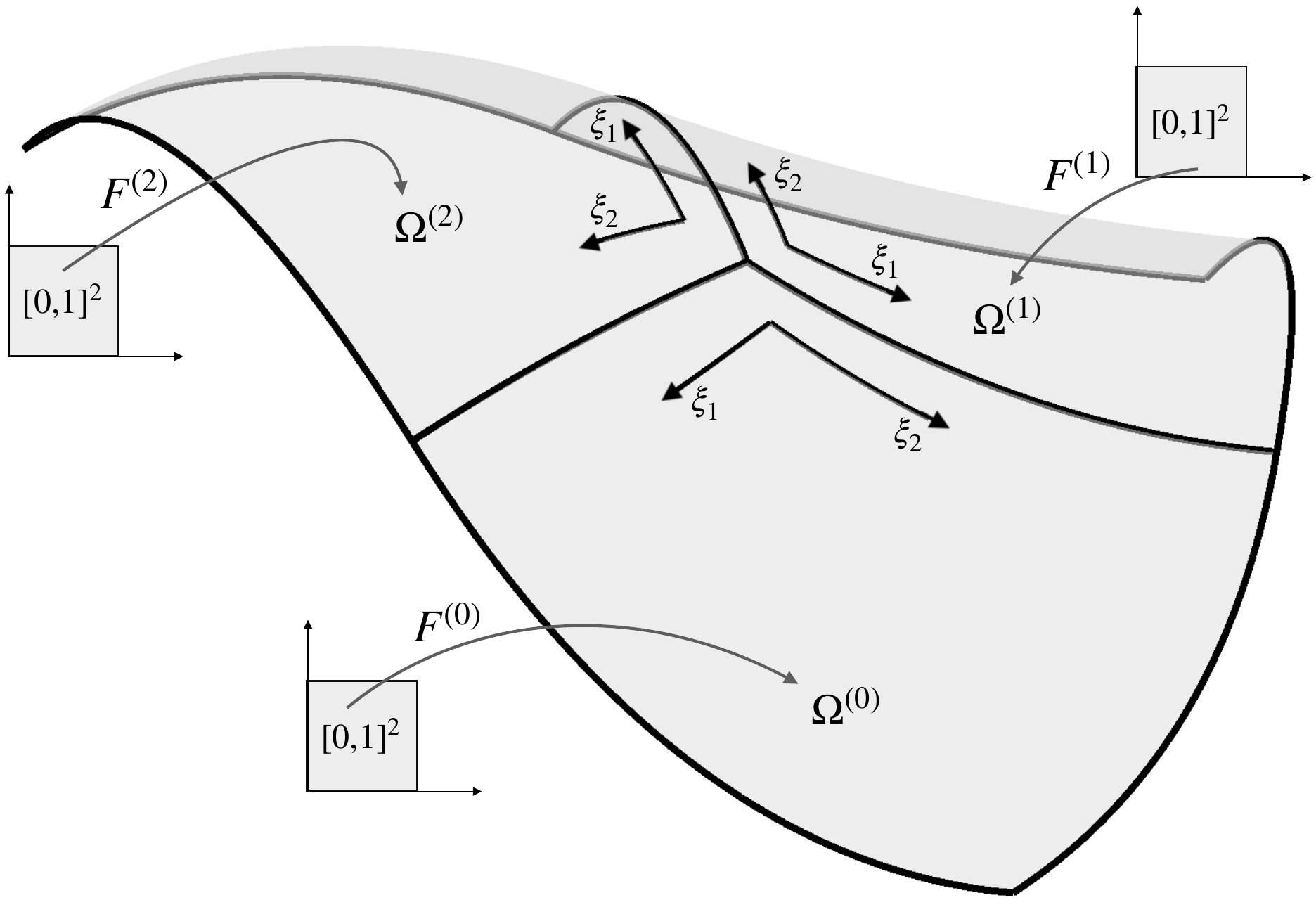}	
  \caption{{\new A multi-patch surface domain $\Omega$ consisting of three surface patches $\Omega^{(i)}$ with the corresponding geometry mappings $\f{F}^{(i)}$, $i=1,2,3$.}
  }
   \label{Domain01}
\end{figure}
The surface domain~$\Omega$ can be further decomposed into the disjoint union
\begin{equation} \label{eq:decomposition}
 \Omega = \left( \bigcup_{i \in \mathcal{I}_{\Omega}} \Omega^{(i)} \right) \;  \dot{\cup} \left( \bigcup_{i \in \mathcal{I}_{\Sigma}} \Sigma^{(i)} \right) \dot{\cup} \left( \bigcup_{i \in \mathcal{I}_{\chi}} \f{x}^{(i)}\right) 
\end{equation}
with the open quadrilateral surface patches~$\Omega^{(i)}$, $i \in \mathcal{I}_{\Omega}$, with open interface and boundary curves~$\Sigma^{(i)}$, $i \in \mathcal{I}_{\Sigma}$, given by the boundary curves of the surface patch parameterizations~$\f{F}^{(j)}$, and with inner and boundary vertices~$\f{x}^{(i)}$, $i \in \mathcal{I}_{\chi}$, given by the corner points of the surface patch parameterizations~$\f{F}^{(j)}$, cf. Fig.~\ref{Domain0}. To distinguish the case of an interface and boundary curve~$\Sigma^{(i)}$, $i \in \mathcal{I}_{\Sigma}$, as well as of an inner and boundary vertex~$\f{x}^{(i)}$, $i \in \mathcal{I}_{\chi}$, we decompose the index sets $\mathcal{I}_{\Sigma}$ and $\mathcal{I}_{\chi}$ into $\mathcal{I}_{\Sigma}^{\circ}$ and $\mathcal{I}_{\Sigma}^{\Gamma}$, and into $\mathcal{I}_{\chi}^{\circ}$ and $\mathcal{I}_{\chi}^{\Gamma}$, respectively. In addition, we denote the patch valency of a vertex~$\f{x}^{(i)}$, $i \in \mathcal{I}_{\chi}$, by $\nu_{i}$. 

\begin{figure}[htb]
	\centering
		\includegraphics[scale=0.35]{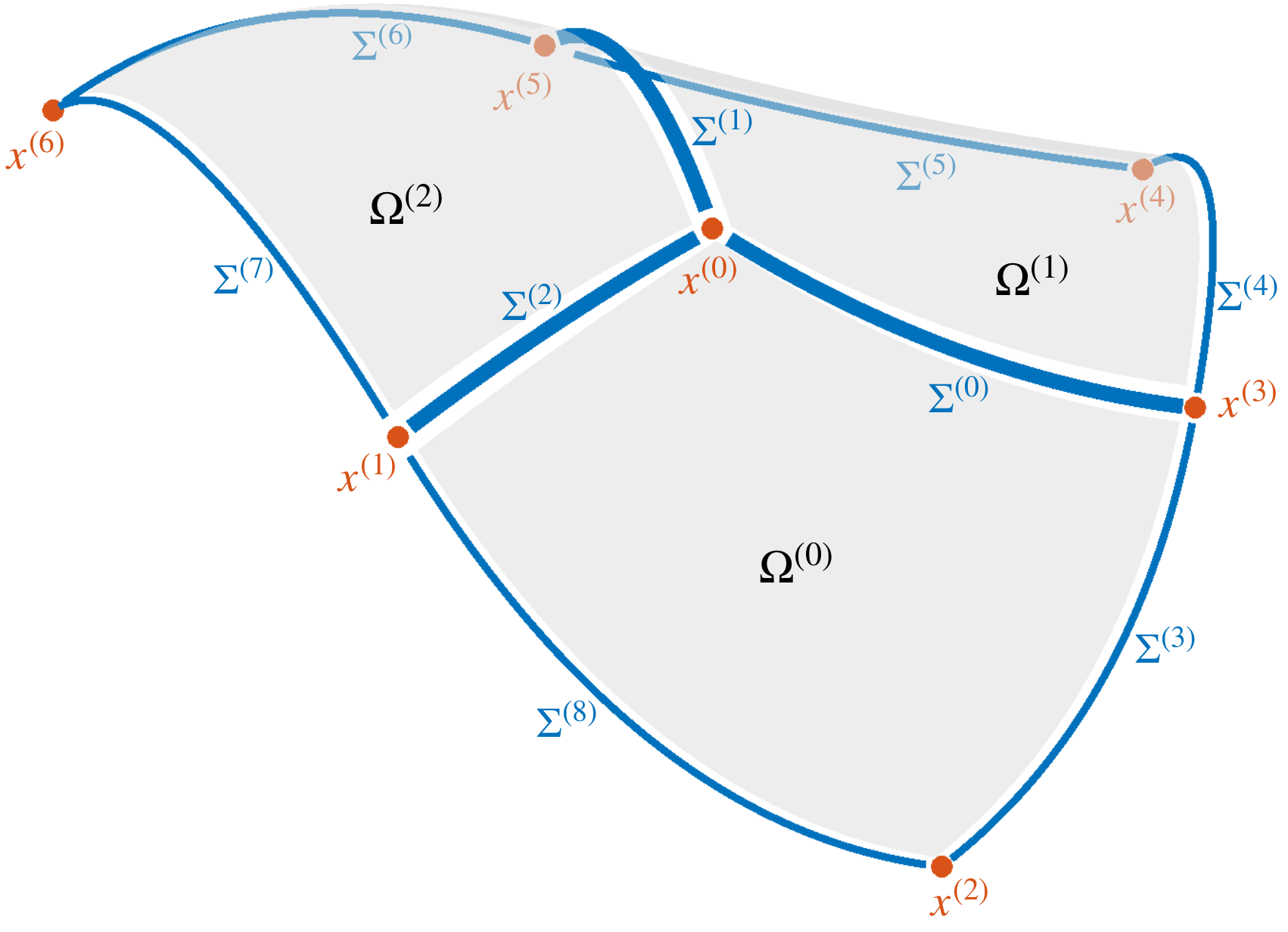}	
  \caption{{\new Decomposition~\eqref{eq:decomposition} of the multi-patch surface domain $\Omega$ into the patches $\Omega^{(i)}$ (gray), curves $\Sigma^{(i)}$ (blue) and vertices $\f{x}^{(i)}$ (red).}}
   \label{Domain0}
\end{figure}

To simplify the notation of the paper, we will assume for each interface curve~$\Sigma^{(i)}$, $i \in \mathcal{I}_{\Sigma}^{\circ}$, with $\Sigma^{(i)} \subset \overline{\Omega^{(i_1)}} \cap \overline{\Omega^{(i_2)}}$, that the two associated surface patch parameterizations are (re)parameterized in such a way that the interface curve~$\Sigma^{(i)}$ is given by
\[
 \f{F}^{(i_1)}(0,\xi) = \f{F}^{(i_2)}(\xi,0), \quad \xi \in (0,1),
\]
see Fig.~\ref{Domain01}, and will assume for each boundary curve~$\Sigma^{(i)}$, $i \in \mathcal{I}_{\Sigma}^{\Gamma}$, with $\Sigma^{(i)} \subset \overline{\Omega^{(i_1)}}$ that the corresponding surface patch parameterization is (re)parameterized such that
\[
 \Sigma^{(i)} = \{\f{F}^{(i_1)}(0,\xi) \, : \, \xi \in (0,1)  \}.
\]
Moreover, we will assume for each inner vertex~$\f{x}^{(i)}$, $i \in \mathcal{I}_{\chi}^{\circ}$, with patch valency~$\nu_{i}$, that the surface patches and interface curves around the vertex~$\f{x}^{(i)}$ are labeled in counterclockwise order as~$\Omega^{(i_1)}$, $\Sigma^{(i_1)}$, $\Omega^{(i_2)}$, $\ldots$, $\Omega^{(i_{\nu_i})}$, $\Sigma^{(i_{\nu_i})}$, and that the associated surface patch parameterizations~$\f{F}^{(i_{\ell})}$, $\ell=1,\ldots,\nu_i$, are (re)parameterized in such a way that the interface curves~$\Sigma^{(i_{\ell})}$, $\ell=1,\ldots,\nu_i-1
$, are given as
\begin{equation} \label{eq:interface_representation}
 \f{F}^{(i_{\ell})}(0,\xi) = \f{F}^{(i_{\ell+1})}(\xi,0), \quad \xi \in (0,1),
\end{equation}
and that the remaining interface curve~$\Sigma^{(i_{\nu_i})}$ is given as
\[
 \f{F}^{(i_{\nu_i})}(0,\xi) = \f{F}^{(i_{1})}(\xi,0), \quad \xi \in (0,1),
\]
which implies that
\begin{equation} \label{eq:vertex_representation}
 \f{x}^{(i)}=\f{F}^{(i_1)}(0,0) = \cdots = \f{F}^{(i_{\nu_i})}(0,0),
\end{equation}
cf. Fig.~\ref{Domain01}. Similarly, we will assume for each boundary vertex~$\f{x}^{(i)}$, $i \in \mathcal{I}_{\chi}^{\Gamma}$, with patch valency~$\nu_{i}$, that the surface patches and interface/boundary curves around the vertex~$\f{x}^{(i)}$ are labeled in counterclockwise order as~$\Sigma^{(i_0)}$, $\Omega^{(i_1)}$, $\Sigma^{(i_1)}$, $\Omega^{(i_2)}$, $\ldots$, $\Omega^{(i_{\nu_i})}$, $\Sigma^{(i_{\nu_i})}$ with the boundary curves~$\Sigma^{(i_0)}$ and $\Sigma^{(i_{\nu_i})}$, and that the associated surface patch parameterizations~$\f{F}^{(i_{\ell})}$, $\ell=1,\ldots,\nu_i$, are (re)parameterized in such a way that the interface curves~$\Sigma^{(i_{\ell})}$, $\ell=1,\ldots,\nu_i-1$, are again given as~\eqref{eq:interface_representation},
and which again implies~\eqref{eq:vertex_representation}.

\subsection{AS-$G^1$ multi-patch spline surfaces} \label{subsec:ASG1_surfaces}

We recall that a multi-patch spline surface~$\f{F}$ is $G^1$-smooth if and only if for each interface curve~$\Sigma^{(i)}$, $i \in \mathcal{I}^{\circ}_{\Sigma}$, with $\Sigma^{(i)} \subset \overline{\Omega^{(i_1)}} \cap \overline{\Omega^{(i_2)}}$, there exist functions $\alpha^{(i,i_1)}:[0,1] \rightarrow \R$, $\alpha^{(i,i_2)}:[0,1]\rightarrow \R$ and $\beta^{(i)}: [0,1] \rightarrow \R$ such that for the associated surface patch parameterizations~$\f{F}^{(i_1)}$ and $\f{F}^{(i_2)}$ and for all $\xi \in [0,1]$ we have
\begin{equation} \label{eq:alphas}
 \alpha^{(i,i_1)}(\xi) \, \alpha^{(i,i_2)} (\xi) >0,
\end{equation}
\begin{equation}  \label{eq:C0continuity}
 \f{f}_0^{(i,i_1)}(\xi) = \f{f}_0^{(i,i_2)}(\xi)
\end{equation}
and
\[
 \alpha^{(i,i_1)}(\xi)\partial_2 \f{F}^{(i_2)}(\xi,0) + \alpha^{(i,i_2)}(\xi) \partial_1 \f{F}^{(i_1)}(0,\xi) + \beta^{(i)}(\xi) \partial_2 \f{F}^{(i_1)}(0,\xi) = \f{0},
\]
with 
\begin{equation*}
\f{f}_0^{(i,i_1)}(\xi)=\f{F}^{(i_1)}(0,\xi) \mbox{ and }\f{f}_0^{(i,i_2)}(\xi)=\f{F}^{(i_2)}(\xi,0),
\end{equation*}
cf.~\cite{HoLa93,Pe02}. It is worth to note that the functions $\alpha^{(i,i_1)}$, $\alpha^{(i,i_2)}$ and $\beta^{(i)}$ are uniquely determined up to a common function~$\gamma^{(i)}: [0,1] \rightarrow \R$. 

A multi-patch spline surface $\f{F}$ is called analysis-suitable $G^1$ (AS-$G^1$) if the surface~$\f{F}$ is $G^1$-smooth, {\new and} if {\new further} for each interface curve~$\Sigma^{(i)}$, $i \in \mathcal{I}^{\circ}_{\Sigma}$, with $\Sigma^{(i)} \subset \overline{\Omega^{(i_1)}} \cap \overline{\Omega^{(i_2)}}$, the functions $\alpha^{(i,i_1)}$ and $\alpha^{(i,i_2)}$ can be {\new selected as} linear polynomials, and 
there exist linear polynomials $\beta^{(i,i_1)}:[0,1] \rightarrow \R$ and $\beta^{(i,i_2)}:[0,1] \rightarrow \R$ such that 
\begin{equation} \label{eq:beta}
 \beta^{(i)}(\xi) = \alpha^{(i,i_1)}(\xi) \,\beta^{(i,i_2)}(\xi) + \alpha^{(i,i_2)}(\xi) \,\beta^{(i,i_1)}(\xi), \quad \xi \in [0,1],
\end{equation}
cf.~\cite{CoSaTa16}. Note that the functions $\beta^{(i,i_1)}$ and $\beta^{(i,i_2)}$ are not uniquely determined. Using Eqs.~\eqref{eq:alphas}, \eqref{eq:C0continuity} and \eqref{eq:beta}, we can reformulate the AS-$G^1$ condition as follows: A multi-patch spline surface~$\f{F}$ is AS-$G^1$ if for each interface curve~$\Sigma^{(i)}$, $i \in \mathcal{I}_{\Sigma}^{\circ}$, with $\Sigma^{(i)} \subset \overline{\Omega^{(i_1)}} \cap \overline{\Omega^{(i_2)}}$, there exist linear polynomials $\alpha^{(i,i_{\ell})}:[0,1] \rightarrow \R$ and $\beta^{(i,i_{\ell})}:[0,1] \rightarrow \R$, $\ell =1,2$, such that Eqs.~\eqref{eq:alphas} and \eqref{eq:C0continuity} are satisfied, and we further have for all $\xi \in [0,1]$
\begin{equation} \label{eq:AS_G1}
 \f{f}_{1}^{(i,i_1)}(\xi) =  \f{f}_{1}^{(i,i_2)}(\xi)
\end{equation}
with
\begin{equation} \label{eq:f1_1}
 \f{f}_{1}^{(i,i_1)}(\xi)= \frac{\partial_1 \f{F}^{(i_1)}(0,\xi)+\beta^{(i,i_1)}(\xi)\partial_2 \f{F}^{(i_1)}(0,\xi)}{\alpha^{(i,i_1)}(\xi)} 
\end{equation}
and
\begin{equation} \label{eq:f1_2}
 \f{f}_1^{(i,i_2)}(\xi)= - \frac{\partial_2 \f{F}^{(i_2)}(\xi,0)+\beta^{(i,i_2)}\partial_1 \f{F}^{(i_2)}(\xi,0)}{\alpha^{(i,i_2)}(\xi)}, 
\end{equation}
where $\f{F}^{(i_1)}$ and $\f{F}^{(i_2)}$ are the associated surface patch parameterizations of $\Omega^{(i_1)}$ and $\Omega^{(i_2)}$, respectively. We denote the equally valued terms~\eqref{eq:C0continuity} and $\eqref{eq:AS_G1}$ by $\f{f}_0^{(i)}$ and $\f{f}_{1}^{(i)}$, respectively, where $\f{f}_0^{(i)}$ just represents the interface curve~$\Sigma^{(i)}$ and $\f{f}_{1}^{(i)}$ describes a specific transversal derivative of the multi-patch spline surface~$\f{F}$ at the interface curve~$\Sigma^{(i)}$. We further denote the functions $\alpha^{(i,i_{\ell})}$ and $\beta^{(i,i_{\ell})}$, $\ell=1,2$, as gluing functions.

\subsection{Concept of $C^1$-smooth isogeometric spline spaces} \label{subsec:C1_spaces}

Let $\f{F}$ be an AS-$G^1$ multi-patch spline surface with its surface patch parameterizations~$\f{F}^{(i)}$, $i \in \mathcal{I}_{\Omega}$. The space of $C^1$-smooth isogeometric spline functions over $\f{F}$ is given as
\[
 \mathcal{V}^1= \{ \phi \in C^1(\Omega) \, : \, \phi \, \circ \f{F}^{(i)} \in \mathcal{S}^{\f{p},\f{r}}_{\f{k}}, i \in \mathcal{I}_{\Omega} \}.
\]
The space~$\mathcal{V}^1$ can be characterized as follows: A function~$\phi \in L^{2}(\Omega)$ belongs to the space~$\mathcal{V}^1$ if and only if for each interface curve~$\Sigma^{(i)}$, $i \in \mathcal{I}_{\Sigma}^{\circ}$, with $\Sigma^{(i)} \subset \overline{\Omega^{(i_1)}} \cap \overline{\Omega^{(i_2)}}$, the function~$\phi$ fulfills for all $\xi \in [0,1]$
\begin{equation} \label{eq:trace}
 \left(\phi \, \circ \f{F}^{(i_1)} \right)(0,\xi) = \left( \phi \, \circ \f{F}^{(i_2)} \right)(\xi,0) 
\end{equation}
and 
\begin{eqnarray}
 \frac{\partial_1 \left(\phi \, \circ \f{F}^{(i_1)} \right)(0,\xi) + \beta^{(i,i_1)}(\xi)\partial_2\left(\phi \, \circ \f{F}^{(i_1)} \right)(0,\xi)}{\alpha^{(i,i_1)}(\xi)} = \nonumber \\ 
 \label{eq:derivative} \\[-0.2cm]
  - \frac{\partial_2 \left(\phi \, \circ \f{F}^{(i_2)} \right)(\xi,0) + \beta^{(i,i_2)}(\xi)\partial_1\left(\phi \, \circ \f{F}^{(i_2)} \right)(\xi,0)}{\alpha^{(i,i_2)}(\xi)}, \nonumber
\end{eqnarray}
where $\f{F}^{(i_1)}$ and $\f{F}^{(i_2)}$ are the associated surface patch parameterizations of $\Omega^{(i_1)}$ and $\Omega^{(i_2)}$, respectively, cf.~\cite{FaJuKaTa22}. The equally valued terms~\eqref{eq:trace} and \eqref{eq:derivative} represent the trace of the function~$\phi$ along the interface curve~$\Sigma^{(i)}$ and a specific transversal derivative of the function~$\phi$ across the interface curve~$\Sigma^{(i)}$, respectively, cf.~\cite{KaSaTa17a}, and will be denoted by $g_{0}^{(i)}$ and $g_{1}^{(i)}$, respectively. In case of a boundary curve~$\Sigma^{(i)}$, i.e. $i \in \mathcal{I}_{\Sigma}^{\Gamma}$, we equivalently define $g_0^{(i)}$ and $g_1^{(i)}$ by taking the left sides of~\eqref{eq:trace} and \eqref{eq:derivative}, and by simply selecting $\alpha^{(i,i_1)}=1$ and $\beta^{(i,i_1)}=0$.

The dimension of $\mathcal{V}^{1}$ heavily depends on the surface patch parameterizations~$\f{F}^{(i)}$, $i \in \mathcal{I}_{\Omega}$, of the AS-$G^1$ multi-patch spline surface~$\f{F}$, cf.~\cite{KaSaTa17a}, which makes the space~$\mathcal{V}^1$ difficult to use. A subspace of $\mathcal{V}^1$, whose dimension is independent of the parameterizations~$\f{F}^{(i)}$  and which possesses optimal polynomial reproduction properties like~$\mathcal{V}^1$, was first introduced for the planar case in \cite{KaSaTa19a} and later extended to the surface case in~\cite{FaJuKaTa22}, and will be employed for the numerical simulation in Section~\ref{subsec:biharmonic_equation}. The $C^1$-smooth subspace is defined as 
\begin{equation} \label{eq:definition}
 \mathcal{A} = \{ \phi \in \mathcal{V}^1 \, : \;  g_0^{(i)} \in \mathcal{S}_{k}^{p,r+1}, \, g_1^{(i)} \in \mathcal{S}_{k}^{p-1,r}, \, i \in \mathcal{I}_{\Sigma}, \mbox{ and } \phi \in C^2_T(\f{x}^{(i)}), \, i \in \mathcal{I}_{\chi}\}, 
\end{equation}
where $C^2_T(\f{x}^{(i)})$ stands for $C^2$-smooth at the vertex~$\f{x}^{(i)}$ with respect to the tangent plane at $\f{x}^{(i)}$. For more details about the space~$\mathcal{A}$, such as the construction of a locally supported basis of~$\mathcal{A}$, we refer to~\cite{FaJuKaTa22}.

\section{Design of AS-$G^1$-multi-patch spline surfaces} \label{sec:design_ASG1_surfaces}

We will present a locally based procedure to generate an AS-$G^1$ multi-patch spline surface from a given $G^1$-smooth but non-AS-$G^1$ multi-patch surface. Before, we will outline the statement of the problem and will explain some precomputations.

\subsection{Problem statement} \label{subsec:problem_statement}

The main objective of the paper is as follows: 
\begin{quote} \em
Given a conforming $G^1$-smooth but non-AS-$G^1$ multi-patch surface~$\f{S}$ consisting of regular quadrilateral surface patch parameterizations~$\f{S}^{(i)}:[0,1]^2 \rightarrow \R^3$, $i \in \mathcal{I}_{\Omega}$, find an AS-$G^1$ multi-patch spline surface~$\f{F}$ with regular quadrilateral surface patch parameterizations~$\f{F}^{(i)} \in (\mathcal{S}_{\f{k}}^{\f{p},\f{r}})^3$, $i \in \mathcal{I}_{\Omega}$, for some appropriate $p$, $r$ and $k$ such that $\f{F}$ possesses the same multi-patch structure as $\f{S}$ and such that $\f{F}$ approximates $\f{S}$ as good as possible by minimizing the error $\| \f{F}-\f{S}\|_{H^1,\sigma}$ with respect to the weighted $H^1$-norm 
\begin{equation} \label{eq:weightedH1error}
\| \f{G}\|_{H^1,\sigma}^2= \sum_{i \in \mathcal{I}_{\Omega}} \sum_{\substack{\ell_1,\ell_2 \in \N_0 \\ \ell_1+\ell_2\leq 1}} \sigma^{\ell_1+\ell_2}  \int_{[0,1]^2} \left\| \partial^{\ell_1}_1\partial^{\ell_2}_2\f{G}^{(i)}(\f{\xi})  \right\|^2  d \f{\xi}
\end{equation}
for a multi-patch surface~$\f{G}$ consisting of the single surface patch parameterizations~$\f{G}^{(i)}$, $i \in \mathcal{I}_{\Omega}$,
where $\sigma$ is a uniform scaling factor (weight) given by $\sigma=\frac{1}{p (k+1)}$.
\end{quote}

Note that the given $G^1$-smooth multi-patch surface~$\f{S}$ need not to be a spline surface. It can be any conforming $G^1$-smooth multi-patch surface with the property that the individual surface patch parameterizations~$\f{S}^{(i)}$, $i \in \mathcal{I}_{\Omega}$, are regular and possess a representation $\f{S}^{(i)}:[0,1]^2 \rightarrow \R^3$. 

Analyzing the AS-$G^1$ condition~\eqref{eq:AS_G1}, one can easily observe: Even for prefixed degree~$p$, regularity~$r$ and number of different inner knots~$k$, the fulfillment of the AS-$G^1$ condition~\eqref{eq:AS_G1} for each interface curve~$\Sigma^{(i)}$, $i \in \mathcal{I}_{\Sigma}^{\circ}$, with $\Sigma^{(i)} \subset \overline{\Omega^{(i_1)}} \cap \overline{\Omega^{(i_2)}}$ is highly non-linear, since the linear gluing functions $\alpha^{(i,i_{\ell})}$ and $\beta^{(i,i_{\ell})}$, $\ell=1,2$, possessing the B\'{e}zier representations
\begin{equation} \label{eq:Bezier_alphas_betas}
\alpha^{(i,i_{\ell})}(\xi) = a_{0}^{(i,i_{\ell})} (1-\xi) + a_{1}^{(i,i_{\ell})} \xi
\quad
\mbox{ and }
\quad
\beta^{(i,i_{\ell})}(\xi) = b_{0}^{(i,i_{\ell})} (1-\xi) + b_{1}^{(i,i_{\ell})} \xi
\end{equation}
with the unknown control points $a_j^{(i,i_{\ell})}, b_j^{(i,i_{\ell})} \in \R$, $j=0,1$, as well as the corresponding surface patch parameterizations~$\f{F}^{(i_{\ell})}$, $\ell=1,2$, possessing the B-spline representations~\eqref{eq:Bspline_representation} with the unknown control points~$\f{d}^{(i_{\ell})}_{j_1,j_2} \in \R^3$ have to be determined. To make in our proposed method in Subsection~\ref{subsec:exact_ASG1_surfaces} the fulfillment of the AS-$G^1$ condition~\eqref{eq:AS_G1} linear, we will precompute for each interface curve~$\Sigma^{(i)}$ the linear gluing functions $\alpha^{(i,i_{\ell})}$ and $\beta^{(i,i_{\ell})}$, $\ell=1,2$. This precomputation as well as the preselection of the degree~$p$, regularity~$r$ and number of different inner knots~$k$ of the spline space~$\mathcal{S}_{\f{k}}^{\f{p},\f{r}}$ will be explained in the following subsection.   

\subsection{Precomputation $\&$ selection of data} \label{subsec:precomputation_data}

We will first describe the precomputation of the {\new linear} gluing functions $\alpha^{(i,i_{\ell})}$ and $\beta^{(i,i_{\ell})}$, $\ell=1,2$, for each interface curve~$\Sigma^{(i)}$ with $\Sigma^{(i)} \subset \overline{\Omega^{(i_1)}} \cap \overline{\Omega^{(i_2)}}$, which 
{\new will reflect the configuration of the first derivatives of the surface patch parameterizations~$\f{S}^{(i)}$, $i \in \mathcal{I}_{\Omega}$, of the given $G^1$-smooth multi-patch surface~$\f{S}$ at the corresponding vertices.}
{\new Thereby,} the proposed technique will estimate the gluing functions from the given
surface~$\f{S}$ and will be motivated by the calculation of the gluing functions for the case of bilinearly parameterized planar multi-patch domains, which are AS-$G^1$ via construction. Let $\widehat{\f{S}}$ be a bilinearly parameterized planar multi-patch domain with regular quadrilateral bilinear patches~$\widehat{\f{S}}^{(i)}$, $i \in \mathcal{I}_{\Omega}$, and interfaces~$\widehat{\Sigma}^{(i)}$, $i \in \mathcal{I}_{\Sigma}^\circ$. For each interface~$\widehat{\Sigma}^{(i)}$ with the two corresponding neighboring patches~$\widehat{\f{S}}^{(i_1)}$ and $\widehat{\f{S}}^{(i_2)}$, the linear gluing functions $\widehat{\alpha}^{(i,i_{\ell})}$ and $\widehat{\beta}^{(i,i_{\ell})}$, $\ell=1,2$, with 
\[
\widehat{\alpha}^{(i,i_{\ell})}(\xi) = \widehat{a_{0}}^{(i,i_{\ell})} (1-\xi) + \widehat{a_{1}}^{(i,i_{\ell})} \xi
\quad
\mbox{ and }
\quad
\widehat{\beta}^{(i,i_{\ell})}(\xi) = \widehat{b_{0}}^{(i,i_{\ell})} (1-\xi) + \widehat{b_{1}}^{(i,i_{\ell})} \xi,
\]
can be simply computed via
\[
 \widehat{a_0}^{(i,i_1)} = \det [\partial_1 \widehat{\f{S}}^{(i,i_1)}(0,0)  , \partial_2 \widehat{\f{S}}^{(i,i_1)}(0,0) ], \quad 
  \widehat{a_1}^{(i,i_1)} = \det [\partial_1 \widehat{\f{S}}^{(i,i_1)}(0,1)  , \partial_2 \widehat{\f{S}}^{(i,i_1)}(0,1) ], 
\]
\[
 \widehat{a_0}^{(i,i_2)} = \det [\partial_1 \widehat{\f{S}}^{(i,i_2)}(0,0)  , \partial_2 \widehat{\f{S}}^{(i,i_2)}(0,0) ], \quad 
  \widehat{a_1}^{(i,i_2)} = \det [\partial_1 \widehat{\f{S}}^{(i,i_2)}(1,0)  , \partial_2 \widehat{\f{S}}^{(i,i_2)}(1,0) ], 
\]
\[
 \widehat{b_0}^{(i,i_1)}= \frac{\partial_1 \widehat{\f{S}}^{(i,i_1)}(0,0) \cdot \partial_2 \widehat{\f{S}}^{(i,i_1)}(0,0) }{||\partial_2 \widehat{\f{S}}^{(i,i_1)}(0,0) ||^2} , \quad \widehat{b_1}^{(i,i_1)}= \frac{\partial_1 \widehat{\f{S}}^{(i,i_1)}(0,1) \cdot \partial_2 \widehat{\f{S}}^{(i,i_1)}(0,1) }{||\partial_2 \widehat{\f{S}}^{(i,i_1)}(0,1) ||^2} 
\]
and
\[
 \widehat{b_0}^{(i,i_2)}= \frac{\partial_1 \widehat{\f{S}}^{(i,i_2)}(0,0) \cdot \partial_2 \widehat{\f{S}}^{(i,i_2)}(0,0) }{||\partial_1 \widehat{\f{S}}^{(i,i_2)}(0,0) ||^2} , \quad \widehat{b_1}^{(i,i_2)}= \frac{\partial_1 \widehat{\f{S}}^{(i,i_2)}(1,0) \cdot \partial_2 \widehat{\f{S}}^{(i,i_2)}(1,0) }{||\partial_1 \widehat{\f{S}}^{(i,i_2)}(1,0) ||^2} ,
\]
cf.~\cite{CoSaTa16}. We see that the linear gluing functions $\widehat{\alpha}^{(i,i_{\ell})}$ and $\widehat{\beta}^{(i,i_{\ell})}$, $\ell=1,2$, are fully determined just via first derivatives of the patches $\widehat{\f{S}}^{(i,i_{\ell})}$, $\ell=1,2$, at the corresponding vertices. Furthermore, we have 
directly via construction that $\widehat{\alpha}^{(i,i_1)} \widehat{\alpha}^{(i,i_2)} > 0$.

We will adapt this design idea to construct the linear gluing functions $\alpha^{(i,i_{\ell})}$ and $\beta^{(i,i_{\ell})}$, $\ell=1,2$, for each interface curve~$\Sigma^{(i)}$, $i \in \mathcal{I}_{\Sigma}^{\circ}$, with $\Sigma^{(i)} \subset \overline{\Omega^{(i_1)}} \cap \overline{\Omega^{(i_2)}}$ for the desired AS-$G^1$ multi-patch spline surface. More precisely, we compute the linear gluing functions $\alpha^{(i,i_{\ell})}$ and $\beta^{(i,i_{\ell})}$, $\ell=1,2$, as before for the linear gluing functions $\widehat{\alpha}^{(i,i_{\ell})}$ and $\widehat{\beta}^{(i,i_{\ell})}$, $\ell=1,2$, but now with respect to the tangent planes at the corresponding vertices and by means of the first derivatives of the surface patch parameterizations $\f{S}^{(i_{\ell})}$, $\ell=1,2$, of the given $G^1$-smooth multi-patch surface~$\f{S}$. This leads to
\[
 a_0^{(i,i_1)} = ||\partial_1 \f{S}^{(i,i_1)}(0,0) \times \partial_2 \f{S}^{(i,i_1)}(0,0) ||, \quad 
  a_1^{(i,i_1)} = ||\partial_1 \f{S}^{(i,i_1)}(0,1) \times \partial_2 \f{S}^{(i,i_1)}(0,1) ||, 
\]
\[
 a_0^{(i,i_2)} = || \partial_1 \f{S}^{(i,i_2)}(0,0)  \times \partial_2 \f{S}^{(i,i_2)}(0,0) ||, \quad 
  a_1^{(i,i_2)} = || \partial_1 \f{S}^{(i,i_2)}(1,0)  \times \partial_2 \f{S}^{(i,i_2)}(1,0) ||, 
\]
\[
 b_0^{(i,i_1)}= \frac{\partial_1 \f{S}^{(i,i_1)}(0,0) \cdot \partial_2 \f{S}^{(i,i_1)}(0,0) }{||\partial_2 \f{S}^{(i,i_1)}(0,0) ||^2} , \quad b_1^{(i,i_1)}= \frac{\partial_1 \f{S}^{(i,i_1)}(0,1) \cdot \partial_2 \f{S}^{(i,i_1)}(0,1) }{||\partial_2 \f{S}^{(i,i_1)}(0,1) ||^2} 
\]
and
\[
 b_0^{(i,i_2)}= \frac{\partial_1 \f{S}^{(i,i_2)}(0,0) \cdot \partial_2 \f{S}^{(i,i_2)}(0,0) }{||\partial_1 {\f{S}}^{(i,i_2)}(0,0) ||^2} , \quad b_1^{(i,i_2)}= \frac{\partial_1 \f{S}^{(i,i_2)}(1,0) \cdot \partial_2 \f{S}^{(i,i_2)}(1,0) }{||\partial_1 \f{S}^{(i,i_2)}(1,0) ||^2} ,
\]
where $a_j^{(i,i_{\ell})}, b_j^{(i,i_{\ell})}$ are the control points of the linear gluing functions $\alpha^{(i,i_{\ell})}$ and $\beta^{(i,i_{\ell})}$, $\ell=1,2$, with B\'{e}zier representations~\eqref{eq:Bezier_alphas_betas}. Again, we directly obtain that $\alpha^{(i,i_1)} \alpha^{(i,i_2)} >0$.  
{\new Moreover,}  
the resulting linear gluing functions {\new reflect the configuration of the first derivatives of the surface patch parameterizations~$\f{S}^{(i)}$, $i \in \mathcal{I}_{\Omega}$, of the given $G^1$-smooth multi-patch surface~$\f{S}$} 
at the corresponding vertices. 

In case {\new that the given surface~$\f{S}$ possesses}  
a $C^1$-$G^1$ transition 
between the first and second ring of patches around an extraordinary vertex, like for G-splines, cf.~Fig.~\ref{GSplineCompareSmall}, left, one further precomputation step, now for the given surface~$\f{S}$, will be necessary. The reason is that for such configurations the inner ring patch gluing functions~$\alpha^{(i,i_j)}$ at the interfaces between the first and second ring {\new cannot be selected as linear polynomials} and have to be at least quadratic, cf.~\cite{Re95}. Enforcing these gluing functions to be linear would imply singularities at the corresponding vertices (Fig.~\ref{GSplineCompareSmall}, left, red vertices). By irregularly subdividing some second ring patches (cf.~Fig.~\ref{GSplineCompareSmall}, middle) allows now the construction of linear gluing functions for all interfaces as demonstrated in more detail in Example~\ref{exGspline}. 

For the proposed method in Section~\ref{subsec:exact_ASG1_surfaces}, we {\new further} preselect {\new for the desired AS-$G^1$ multi-patch spline surface~$\f{F}$ with the individual surface patch parameterizations~$\f{F}^{(i)} \in (\mathcal{S}^{\f{p},\f{r}}_{\f{k}})^3$ , $i \in \mathcal{I}_{\Omega}$}, the degree~$p$, regularity~$r$ and number of different inner knots~$k$ of the used spline space~$\mathcal{S}^{\f{p},\f{r}}_{\f{k}}$.  In any case, we require that $p \geq 3$, $1 \leq r \leq p-2$ and $k \geq \frac{5-p}{p-r-1}$ to have enough degrees of freedom in the vicinity of interface curves~$\Sigma^{(i)}$, $i \in \mathcal{I}_{\Sigma}^{\circ}$, and vertices~$\f{x}^{(i)}$, $i \in \mathcal{I}_{\chi}$, to separate the problem into local minimization problems in Section~\ref{subsec:exact_ASG1_surfaces}. When the given $G^1$-smooth multi-patch surface~$\f{S}$ is a spline surface with spline surface patch parameterizations~$\f{S}^{(i)} \in \left(\mathcal{S}^{\widetilde{\f{p}},\widetilde{\f{r}}}_{\widetilde{\f{k}}}\right)^3$, $\widetilde{\f{p}}=(\widetilde{p},\widetilde{p})$, $\widetilde{\f{r}}=(\widetilde{r},\widetilde{r})$ and $\widetilde{\f{k}}=(\widetilde{k},\widetilde{k})$, we additionally require $p \geq \widetilde{p}$, $r \leq \widetilde{r}$ and $k =2^{\ell}(\widetilde{k}+1)-1$, $\ell \in \N_0$. As long as the quality of the resulting multi-patch spline surface~$\f{F}$ is not good enough, we increase the degree~$p$ and/or reduce the regularity~$r$ and/or increase the number of inner knots~$k$ of the used spline space~$\mathcal{S}_{\f{k}}^{\f{p},\f{r}}$.
 
\subsection{Local design of AS-$G^1$ multi-patch spline surfaces} \label{subsec:exact_ASG1_surfaces}

We will present a locally based Lagrange multiplier method to generate an AS-$G^1$ multi-patch spline surface~$\f{F}$ from a given $G^1$-smooth but non-AS-$G^1$ multi-patch surface~$\f{S}$. The AS-$G^1$ multi-patch spline surface~$\f{F}$ will be generated by finding the control points~$\f{d}_{j_1,j_2}^{(i)}$ of the B-spline representations~\eqref{eq:Bspline_representation} of the individual surface patch parameterizations~$\f{F}^{(i)}$, $i \in \mathcal{I}_{\Omega}$, via solving a series of some subsequently but also of some parallel solvable small minimization problems.

The following lemma characterizes an AS-$G^1$ multi-patch spline surface~$\f{F}$ in the vicinity of an interface curve~$\Sigma^{(i)}$, $i \in \mathcal{I}^{\circ}_{\Sigma}$, and will provide the basis for the introduction of the method.
\begin{lem} \label{lem:surface_vicinity}
Let $\f{F}$ be an AS-$G^1$ multi-patch spline surface with the surface patch parameterizations~$\f{F}^{(i)}$, $i \in \mathcal{I}_{\Omega}$. For each interface curve~$\Sigma^{(i)}$, $i \in \mathcal{I}_{\Sigma}^{\circ}$, with $\Sigma^{(i)} \subset \overline{\Omega^{(i_1)}} \cap \overline{\Omega^{(i_2)}}$, the associated surface patch parameterizations~$\f{F}^{(i_1)}$ and $\f{F}^{(i_2)}$ possess the representation
\[
\f{F}^{(i_1)}(\xi_1,\xi_2) = \f{f}_{0}^{(i)}(\xi_2) + \xi_1 \left( \alpha^{(i,i_1)}(\xi_2) \f{f}_{1}^{(i)}(\xi_2) - \beta^{(i,i_1)}(\xi_2) \left(\f{f}_0^{(i)}\right)'(\xi_2) \right) + \mathcal{O}(\xi_1^2)
\]
and
\[
 \f{F}^{(i_2)}(\xi_1,\xi_2) = \f{f}_{0}^{(i)}(\xi_1) + \xi_2 \left( -\alpha^{(i,i_2)}(\xi_1) \f{f}_{1}^{(i)}(\xi_1) - \beta^{(i,i_2)}(\xi_1)\left(\f{f}_{0}^{(i)} \right)'(\xi_1) \right) + \mathcal{O}(\xi_2^2)
\]
in the vicinity of the interface curve~$\Sigma^{(i)}$. 
\end{lem}
\begin{proof}
Let $\Sigma^{(i)}$, $i \in \mathcal{I}_{\Sigma}^{\circ}$, with $\Sigma^{(i)} \subset \overline{\Omega^{(i_1)}} \cap \overline{\Omega^{(i_2)}}$, be an arbitrary but fixed interface curve. Due to 
\begin{equation} \label{eq:equalityf0}
\f{f}_{0}^{(i)}(\xi)=\f{F}^{(i,i_{1})}(0,\xi)=\f{F}^{(i,i_{2})}(\xi,0)
\end{equation}
and $\f{f}_{1}^{(i)}=\f{f}_{1}^{(i,i_{\ell})}$, $\ell =1,2$, and by considering Eqs.~\eqref{eq:f1_1} and \eqref{eq:f1_2}, we obtain
\begin{equation} \label{eq:partial1F1}
 \partial_1 \f{F}^{(i_1)}(0,\xi_2)= \alpha^{(i,i_1)}(\xi_2) \f{f}_{1}^{(i)}(\xi_2) - \beta^{(i,i_1)}(\xi_2) \left(\f{f}_0^{(i)}\right)'(\xi_2)
\end{equation}
and
\begin{equation} \label{eq:partial2F2}
 \partial_2 \f{F}^{(i_2)}(\xi_1,0) =  -\alpha^{(i,i_2)}(\xi_1) \f{f}_{1}^{(i)}(\xi_1) - \beta^{(i,i_2)}(\xi_1)\left(\f{f}_{0}^{(i)} \right)'(\xi_1).
\end{equation}
By means of Taylor expansion of $\f{F}^{(i_{\ell})}$, $\ell=1,2$, at the interface~$\Sigma^{(i)}$, we get
\[
 \f{F}^{(i_1)}(\xi_1,\xi_2)= \f{F}^{(i_1)}(0,\xi_2) + \partial_1 \f{F}^{(i_1)} (0,\xi_2) \xi_1 + \mathcal{O}(\xi_1^2)
\]
and
\[
 \f{F}^{(i_2)}(\xi_1,\xi_2)= \f{F}^{(i_2)}(\xi_1,0) + \partial_2 \f{F}^{(i_2)}(\xi_1,0) \xi_2 + \mathcal{O}(\xi_2^2),
\]
which proves the lemma by using the Eqs.~\eqref{eq:equalityf0}, \eqref{eq:partial1F1} and \eqref{eq:partial2F2}.
\end{proof}
Due to Lemma~\ref{lem:surface_vicinity}, we will enforce for each interface curve~$\Sigma^{(i)}$, $i \in \mathcal{I}_{\Sigma}^\circ$, that $\f{f}_{0}^{(i)} \in \left(\mathcal{S}^{p,r+1}_{k}\right)^3$ and $\f{f}_{1}^{(i)} \in \left(\mathcal{S}^{p-1,r}_{k}\right)^3$, which means that the functions $\f{f}_0^{(i)}$ and $\f{f}_{1}^{(i)}$ possess the B-spline representation
\begin{equation} \label{eq:repf0}
 \f{f}_{0}^{(i)}(\xi) = \sum_{j \in \mathcal{J}_0} \widehat{\f{d}}_{0,j}^{(i)} N_{j}^{p,r+1}(\xi), \quad \widehat{\f{d}}_{0,j}^{(i)} \in \R^3,
\end{equation}
and
\begin{equation} \label{eq:repf1}
  \f{f}_{1}^{(i)}(\xi) = \sum_{j \in \mathcal{J}_1} \widehat{\f{d}}_{1,j}^{(i)} N_{j}^{p-1,r}(\xi), \quad \widehat{\f{d}}_{1,j}^{(i)} \in \R^3,
\end{equation}
respectively.

In our proposed local method, {\new we first precompute for the desired AS-$G^1$ multi-patch spline surface~$\f{F}$ with the individual surface patch parameterizations~$\f{F}^{(i)} \in (\mathcal{S}^{\f{p},\f{r}}_{\f{k}})^3$, $i \in \mathcal{I}_{\Omega}$, the linear gluing functions $\alpha^{(i,i_{\ell})}$ and $\beta^{(i,i_{\ell})}$, $\ell=1,2$, for each interface curve~$\Sigma^{(i)}$ with $\Sigma^{(i)} \subset \overline{\Omega^{(i_1)}} \cap \overline{\Omega^{(i_2)}}$, and further preselect the degree~$p$, regularity~$r$ and number of different inner knots~$k$ of the used spline space~$\mathcal{S}^{\f{p},\f{r}}_{\f{k}}$.} {\new Then,} we determine the control points~$\f{d}^{(i)}_{j_1,j_2}$ {\new of the individual surface patch parameterizations~$\f{F}^{(i)}$} in three consecutive steps, namely by first solving for them a minimization problem for each inner vertex~$\f{x}^{(m)}$, $m \in \mathcal{I}^{\circ}_{\chi}$, and for each boundary vertex~$\f{x}^{(m)}$, $m \in \mathcal{I}^{\Gamma}_{\chi}$, with patch valency~$\nu_m \geq 2$, then one for each interface curve~$\Sigma^{(m)}$, $m \in \mathcal{I}^{\circ}_{\Sigma}$, and finally one for each surface patch~$\Omega^{(i)}$, $i \in \mathcal{I}_{\Omega}$. In doing so, the resulting control points~$\f{d}_{j_1,j_2}^{(i)}$ from one step will be employed in the following steps. The idea is to locally minimize the weighted $H^1$-error~\eqref{eq:weightedH1error} for the vertices, interface curves and patches and to additionally enforce by using Lagrange multipliers that the AS-$G^1$ conditions are fulfilled and that the functions~$\f{f}_0^{(i)}$ and $\f{f}_{1}^{(i)}$ for each interface curve~$\Sigma^{(i)}$ possess the B-spline representations~\eqref{eq:repf0} and \eqref{eq:repf1}, respectively. In detail, we minimize for each inner vertex~$\f{x}^{(m)}$, $m \in \mathcal{I}^{\circ}_{\chi}$, and for each boundary vertex~$\f{x}^{(m)}$, $m \in \mathcal{I}^{\Gamma}_{\chi}$, with patch valency~$\nu_m \geq 2$, the objective function
\begin{eqnarray*}
& & \sum_{j=1}^{\nu_m}\sum_{\substack{\ell_1,\ell_2 \in \N_0 \\ 0<\ell_1+\ell_2\leq 2}} \sigma^{\ell_1+\ell_2} \left\|\partial^{\ell_1}_1\partial^{\ell_2}_2\f{F}^{(m_{j})}(0,0)  - \partial^{\ell_1}_1\partial^{\ell_2}_2\f{S}^{(m_{j})}(0,0) \right\|^2 + \\
&  & \sum_{j=1}^{\nu_m-\tau_m}  
\sum_{\ell=0}^1 \sum_{w=0}^{2-\ell}  \sigma^{w+\ell} \f{\lambda}_{\ell,w}^{(m,j)} \cdot \left(\partial_{\xi}^{w} \f{f}_\ell^{(j,m_{j})}(0) -  \partial_{\xi}^{w}\f{f}_\ell^{(j,m_{j+1})}(0) \right) 
\end{eqnarray*}
for the corresponding control points~$\f{d}_{j_1,j_2}^{(i)}$ and Lagrange multipliers~$\f{\lambda}_{\ell,w}^{(m,j)} = 
\left( {\lambda}_{\ell,w,i}^{(m,j)} \right)_{i=1}^3$, and for each interface curve~$\Sigma^{(m)} \subset \overline{\Omega^{(m_1)}} \cap \overline{\Omega^{(m_2)}}$, $m \in \mathcal{I}^{\circ}_{\Sigma}$, the objective function

\begin{eqnarray*}
& &  \sum_{\substack{\ell_1,\ell_2 \in \N_0\\ \ell_1+\ell_2\leq 1}} \sigma^{\ell_1+\ell_2} \int_{[0,1]} \left\|\partial^{\ell_1}_1\partial^{\ell_2}_2\f{F}^{(m_1)}(0,\xi)  - \partial^{\ell_1}_1\partial^{\ell_2}_2\f{S}^{(m_1)}(0,\xi) \right\|^2  d\xi + \\
& & \sum_{\substack{\ell_1,\ell_2 \in \N_0 \\ \ell_1+\ell_2\leq 1}} \sigma^{\ell_1+\ell_2}  \int_{[0,1]} \left\|\partial^{\ell_1}_1\partial^{\ell_2}_2\f{F}^{(m_2)}(\xi,0)  - \partial^{\ell_1}_1\partial^{\ell_2}_2\f{S}^{(m_2)}(\xi,0) \right\|^2  d\xi + \\
& & \sum_{\ell=1}^2 \left( \sum_{j=3(\ell-1)}^{ n-3\ell+2} \f{\mu}_{0,j}^{(m,m_\ell)}\cdot \left( \f{f}_0^{(m,m_\ell)}( \zeta_j^{p,r})-\f{f}_0^{(m)}( \zeta_j^{p,r}) \right) + \right.  \\ 
& & \left. \sum_{j=2(\ell-1)}^{ n-2\ell+1}  \sigma \,\alpha^{(m,m_\ell)}( \zeta_j^{p,r}) \, \f{\mu}_{1,j}^{(m,m_\ell)} \cdot\left( \f{f}_1^{(m,m_\ell)}( \zeta_j^{p,r})-\f{f}_1^{(m)}( \zeta_j^{p,r}) \right)  \right)  \\ 
\end{eqnarray*}
for the involved control points~$\f{d}_{j_1,j_2}^{(i)}$, $\widehat{\f{d}}_{0,j}^{(i)}$ and $\widehat{\f{d}}_{1,j}^{(i)}$, and Lagrange multipliers~$\f{\mu}_{\ell,j}^{(m,m_{w})} = \left( {\mu}_{\ell,j,i}^{(m,m_{w})}\right)_{i=1}^3$, and for each surface patch~$\Omega^{(i)}$, $i \in \mathcal{I}_{\Omega}$, the term
\[
\sum_{\substack{\ell_1,\ell_2 \in \N_0 \\ \ell_1+\ell_2\leq 1}} \sigma^{\ell_1+\ell_2} \int_{[0,1]^2} 
 \left\|\partial^{\ell_1}_1\partial^{\ell_2}_2\f{F}^{(i)}(\f{\xi})  - \partial^{\ell_1}_1\partial^{\ell_2}_2\f{S}^{(i)}(\f{\xi}) \right\|^2  d \f{\xi}
\]
for the corresponding control points~$\f{d}_{j_1,j_2}^{(i)}$, where $\sigma$ is the uniform scaling factor given by $\sigma=\frac{1}{p (k+1)}$ as already described in the problem statement, and $\tau_m$ is equal to
\begin{equation} \label{eq:tau}
\tau_m =\begin{cases} 
         0 &  \mbox{ if }\f{x}^{(m)} \mbox{ is an inner vertex, i.e. }m \in \mathcal{I}_{\chi}^{\circ}, \\
         1 &  \mbox{ if }\f{x}^{(m)} \mbox{ is a boundary vertex, i.e. }m \in \mathcal{I}_{\chi}^{\Gamma}.
        \end{cases}
\end{equation}
Note that the individual minimization problems in each step can be solved in parallel.   

\begin{rem} \label{remarkGlobal}
Inspired by the local method above, we can also derive a novel global technique which will perform all single local steps just in one merged and slightly adapted global step. The idea is to construct the AS-$G^1$ multi-patch spline surface~$\f{F}$ by minimizing the objective function
\begin{eqnarray*}
& &\sum_{i \in \mathcal{I}_{\Omega}} \left(  \sum_{\substack{\ell_1,\ell_2 \in \N_0 \\ \ell_1+\ell_2\leq 1}} \sigma^{\ell_1+\ell_2} \int_{[0,1]^2} \left\| \partial^{\ell_1}_1\partial^{\ell_2}_2\f{F}^{(i)}(\f{\xi})  - \partial^{\ell_1}_1\partial^{\ell_2}_2\f{S}^{(i)}(\f{\xi}) \right\|^2  d \f{\xi} \right) + \\
&  & \sum_{\substack{m \in \mathcal{I}_{\chi} \\ \nu_{m} \geq 2}} \sum_{j=1}^{\nu_m-\tau_m}  
\sum_{\ell=0}^1 \sum_{w=0}^{2-\ell}  \sigma^{w+\ell} \f{\lambda}_{\ell,w}^{(m,j)} \cdot \left(\partial_{\xi}^{w} \f{f}_\ell^{(j,m_{j})}(0) -  \partial_{\xi}^{w}\f{f}_\ell^{(j,m_{j+1})}(0) \right)   +\\ 
& & \sum_{m \in \mathcal{I}_{\Sigma}^{\circ}} \sum_{\ell=1}^2 \left( \sum_{j=3(\ell-1)}^{ n-3\ell+2} \f{\mu}_{0,j}^{(m,m_\ell)}\cdot \left( \f{f}_0^{(m,m_\ell)}( \zeta_j^{p,r})-\f{f}_0^{(m)}( \zeta_j^{p,r}) \right) + \right.  \\ 
& & \left. \sum_{j=2(\ell-1)}^{ n-2\ell+1}  \sigma \,\alpha^{(m,m_\ell)}( \zeta_j^{p,r}) \, \f{\mu}_{1,j}^{(m,m_\ell)} \cdot\left( \f{f}_1^{(m,m_\ell)}( \zeta_j^{p,r})-\f{f}_1^{(m)}( \zeta_j^{p,r}) \right)  \right)
\end{eqnarray*}
for the control points~$\f{d}_{j_1,j_2}^{(i)}$ of the single surface patch parameterizations~$\f{F}^{(i)}$, $i \in \mathcal{I}_{\Omega}$, and for the Lagrange multipliers $\f{\lambda}_{\ell,w}^{(m,j)} = 
\left( {\lambda}_{\ell,w,i}^{(m,j)} \right)_{i=1}^3$ and $\f{\mu}_{\ell,j}^{(m,m_{w})} = \left( {\mu}_{\ell,j,i}^{(m,m_{w})}\right)_{i=1}^3$, where $\sigma$ and $\tau_m$ are given as in the local approach above.

The presented global approach is like the local method applicable to any $G^1$-smooth but non-AS-$G^1$ multi-patch surface. It can be handled without the need of symbolic computation, and can be seen therefore also as an extension of the two existing global techniques~\cite{FaJuKaTa22,KaSaTa17b}. The presented global method leads to a slightly reduced approximation error compared to the local method, as shown in Example~\ref{exNintendo}, {\new since the global method computes all control points in one step in contrast to three consecutive steps which provides more freedom to determine the control points}. However, this is achieved at the price of solving one big global minimization problem instead of small local ones which can be done in parallel, and makes the global technique only practicable for multi-patch surfaces with a relatively small number of patches.  
\end{rem}

\section{Numerical examples} \label{sec:numerical_examples}

We will first present several examples of AS-$G^1$ multi-patch spline surfaces constructed by the proposed locally based technique in Section~\ref{subsec:exact_ASG1_surfaces}. Afterwards, we will demonstrate that the resulting AS-$G^1$ multi-patch spline surfaces are especially suited for the solving of fourth order partial differential equations over them on the basis of the biharmonic problem.

\subsection{Examples of AS-$G^1$ multi-patch spline surfaces} \label{subsec:examples_ASG1_surfaces}

While the first two examples will deal with planar multi-patch domains, the remaining tree examples will construct AS-$G^1$ approximants for given $G^1$-smooth but non-AS-$G^1$ multi-patch surfaces. In all examples, we will compare the resulting AS-$G^1$ multi-patch surface~$\f{F}$ with the given $G^1$-smooth but non-AS-$G^1$ multi-patch surface~$\f{S}$ by computing the relative $L^2$-error $\epsilon_{L^2}$ 
and the relative weighted $H^1$-error $\epsilon_{H^1}$ given by 
\[
\epsilon_{L^2} = \frac{\left\|\f{F}  - \f{S} \right\|_{L^2}}{\left\|\f{S} \right\|_{L^2}} \quad\mbox { and } \quad \epsilon_{H^1} = \frac{\left\|\f{F}  - \f{S} \right\|_{H^1,\sigma}}{\left\|\f{S} \right\|_{H^1,\sigma}},
\]
respectively, where $\left\|\f{\cdot} \right\|_{H^1,\sigma}$ is the weighted $H^1$-norm defined in \eqref{eq:weightedH1error}.

\paragraph{AS-$G^1$ planar multi-patch geometries}

We will present two examples of AS-$G^1$ planar multi-patch geometries constructed from given non-AS-$G^1$ planar multi-patch parameterizations. 
\begin{ex}\label{exBunny}
We consider the given non-AS-$G^1$ planar multi-patch geometry~$\f{S}$ depicted in Fig.~\ref{bunnyCompare}, left, which is composed of ten bicubic Bézier patches. We approximate the parameterization $\f{S}$ by an AS-$G^1$ multi-patch geometry~$\f{F}$, see Fig.~\ref{bunnyCompare}, middle, which uses splines from the space $(\mathcal{S}^{\f{3},\f{1}}_{\f{2}})^2$ to represent the single patches~$\f{F}^{(i)}$, $i \in \mathcal{I}_{\Omega}$. Fig.~\ref{bunnyCompare}, right, provides {\new a comparison} between the given non-AS-$G^1$ multi-patch geometry~$\f{S}$ and its AS-$G^1$ approximant~$\f{F}$. The two multi-patch geometries exhibit a great similarity, which is confirmed by the relative $L^2$- and $H^1$-error 
equal to $\epsilon_{L^2} = 9.7 \cdot 10^{-4}$ and $\epsilon_{H^1} =5.5 \cdot 10^{-3}$, respectively. 

\begin{figure}[h!]
	\centering
        \includegraphics[scale=0.385]{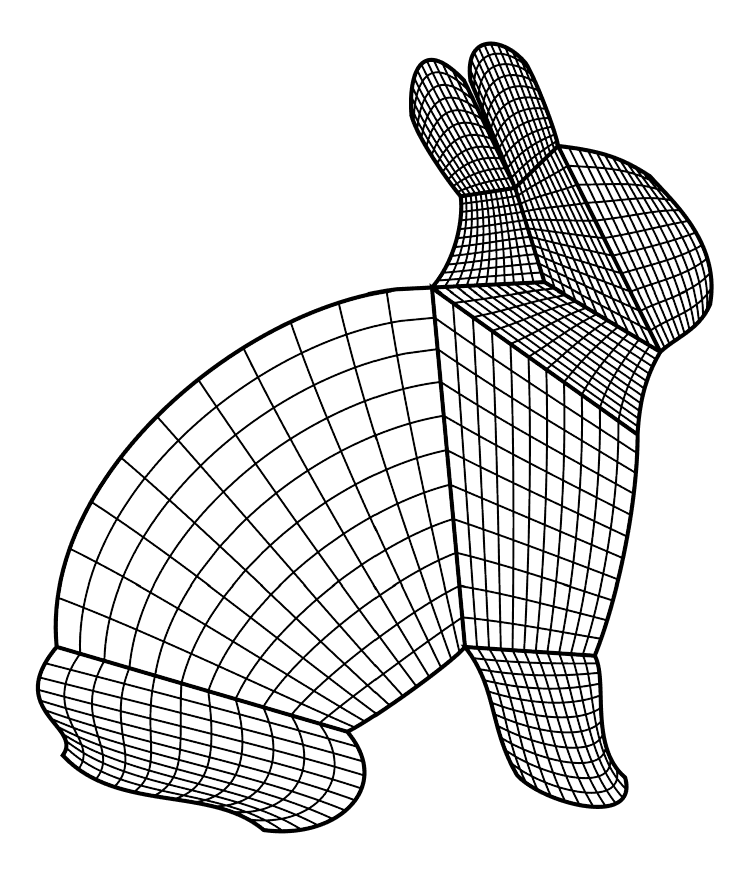}
		\includegraphics[scale=0.385]{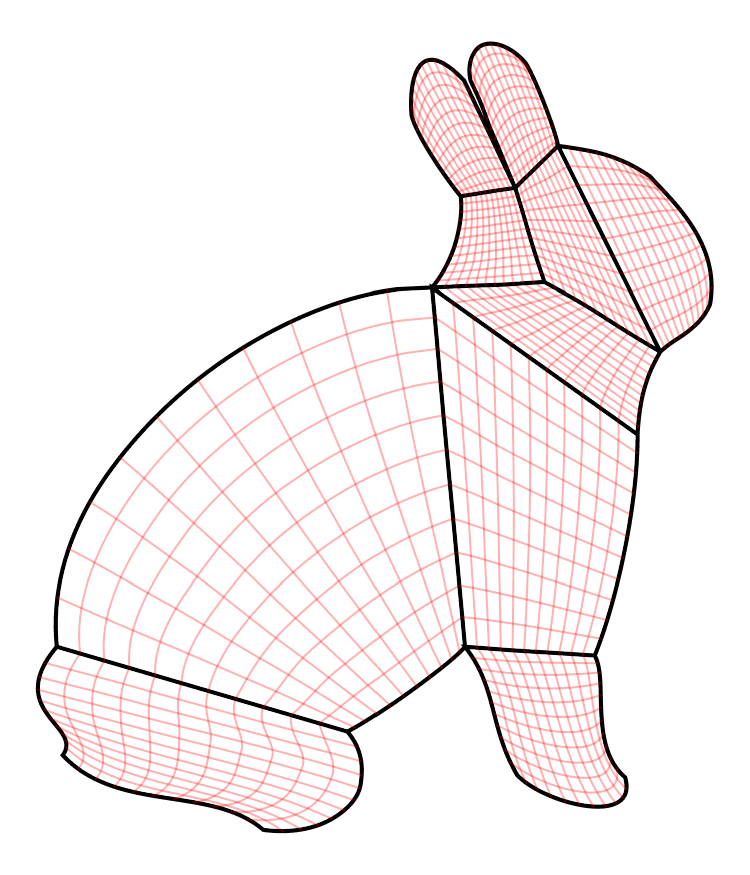}
        \includegraphics[scale=0.48]{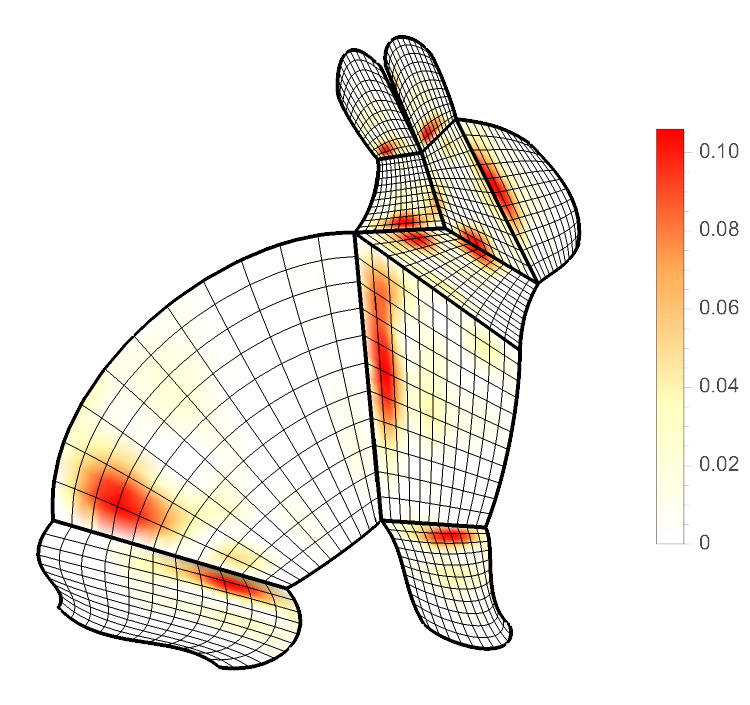}		
        \caption{Example~\ref{exBunny}. The non-AS-$G^1$ planar multi-patch geometry~$\f{S}$ (left), its AS-$G^1$ approximant~$\f{F}$ (middle) and {\new the differences }
        between them (right).}
        \label{bunnyCompare}
\end{figure}
\end{ex}

\begin{ex}\label{exNintendo}
In the second planar example, we approximate the given non-AS-$G^1$ planar 13-patch bicubic Bézier
parameterization~$\f{S}$ (Fig.~\ref{nintendoCompare}, left) by an AS-$G^1$ planar multi-patch geometry~$\f{F}$ (Fig.~\ref{nintendoCompare}, middle) by employing the spline space {\new $(\mathcal{S}^{\f{4},\f{1}}_{\f{2}})^2$} for parameterizing the single patches ~$\f{F}^{(i)}$, $i \in \mathcal{I}_{\Omega}$. As in Example \ref{exBunny}, {\new a comparison} (Fig.~\ref{nintendoCompare}, right) between the two geometries reveals a very good approximation, validated by the relative $L^2$- and $H^1$-error equal to {\new $\epsilon_{L^2} = 1.2 \cdot 10^{-3}$ and $\epsilon_{H^1} = 3.4 \cdot 10^{-3}$}, respectively. 

For {\new further} comparison, we also perform our global algorithm from Remark~\ref{remarkGlobal}. The resulting AS-$G^1$ multi-patch geometry is 
{\new very close to} the one obtained by the local approach, and is therefore not presented here. But as expected, the presented global method leads to slightly reduced approximation errors compared to the local method, namely to a relative $L^2$-error {\new $\epsilon_{L^2} = 1.2 \cdot 10^{-5}$ and to a relative $H^1$-error $\epsilon_{H^1} = 2.3 \cdot 10^{-3}$}. Recall that this is achieved at the price of solving one big global minimization problem instead of small local ones. 
\begin{figure}[h!]
	\centering
        \includegraphics[scale=0.405]
        {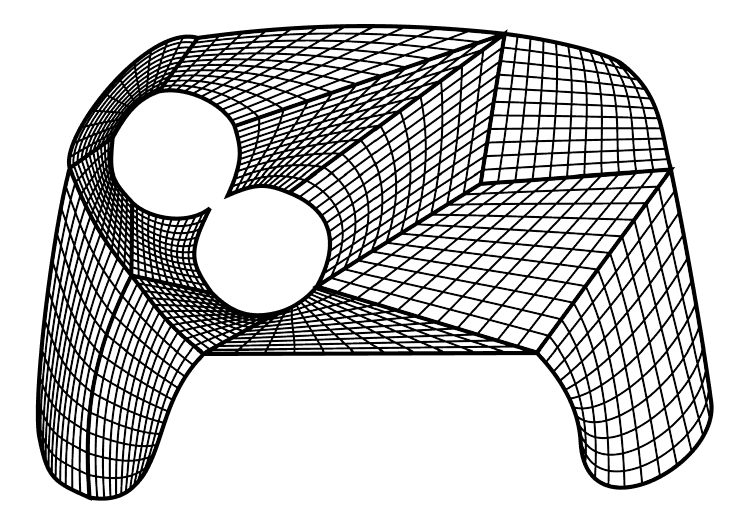}
		\includegraphics[scale=0.297]{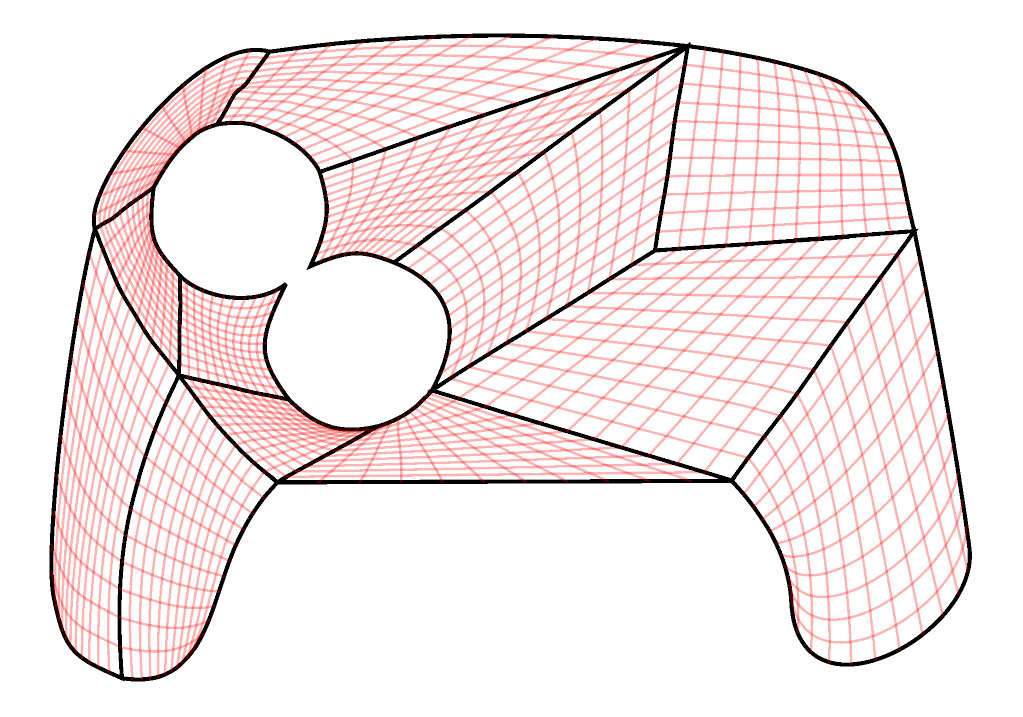}
		\includegraphics[scale=0.29]{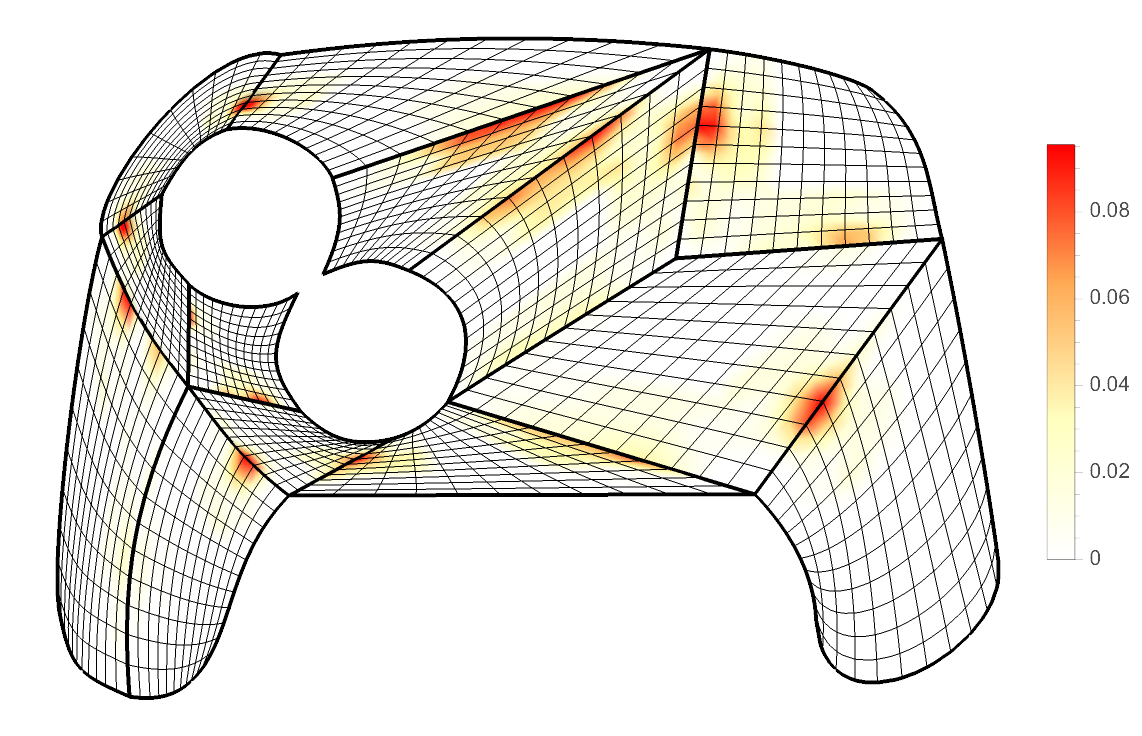}	
  \caption{Example~\ref{exNintendo}. The non-AS-$G^1$ planar multi-patch geometry~$\f{S}$ (left), its AS-$G^1$ approximant~$\f{F}$ (middle) and {\new the differences }
  between them (right).}
  \label{nintendoCompare}
\end{figure}
{\new 
In order to emphasize the time efficiency of the local method over the global one, we compare the CPU computational times of both methods for the 13-patch domain above as well as additionally for a 52-patch domain which is obtained from the 13-patch domain 
by subdividing each of the 13 patches into four new ones, cf.~Fig.~\ref{nintendoCompare2}.
 We have used an implementation in {\em Mathematica} on a M1 Pro processor with 10 cores. 
For the 13-patch domain the local approach reduces the required time by a factor $23.0$, while for the 52-patch domain the reduction is already by a factor 
$34.7$, see Tab.~\ref{tab:tab1}. These factors could be even improved if for the local method the work would be distributed within even more processors/cores.  
Again, for the 52-patch domain the global method leads to slightly reduced approximation errors compared to the local method, namely the relative $L^2$-errors are $\epsilon_{L^2} = 8.9 \cdot 10^{-6}$ and $\epsilon_{L^2} = 3.7 \cdot 10^{-4}$, respectively, while the relative $H^1$-errors are $\epsilon_{H^1} = 1.3 \cdot 10^{-3}$ and $\epsilon_{H^1} = 1.8 \cdot 10^{-3}$, respectively.}
\begin{table}[htb]
    \centering
    {\new 
    \begin{tabular}{|c||c|c|}
         \hline
         & local method & global method \\
        \hline
        \hline
        13-patch domain & 34.15 s & 785.68 s\\
        \hline
        52-patch domain & 134.32 s & 4660.58 s\\
        \hline
    \end{tabular}
    \caption{Example~\ref{exNintendo}. CPU computational times for the 13-patch and 52-patch domains for our local and global method.}
    }
    \label{tab:tab1}
\end{table}
\begin{figure}[h!]
	\centering
		\includegraphics[scale=0.31]{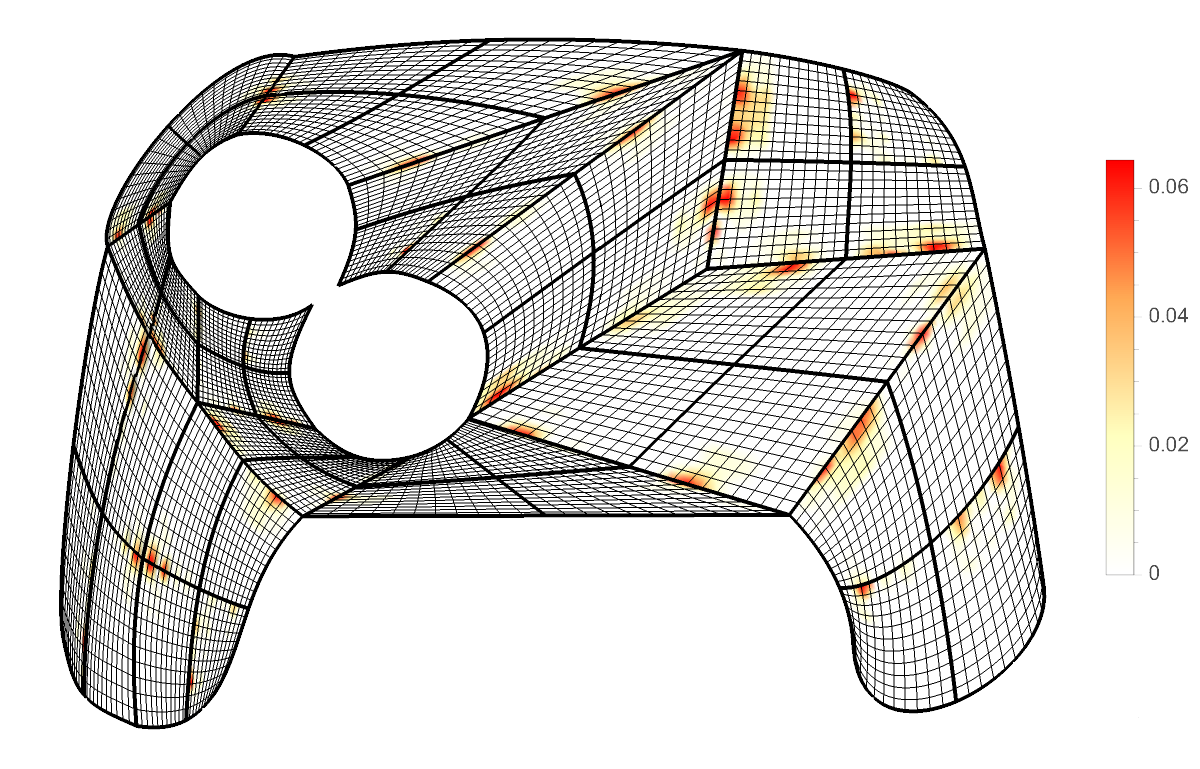}
        \hskip3em
		\includegraphics[scale=0.335]{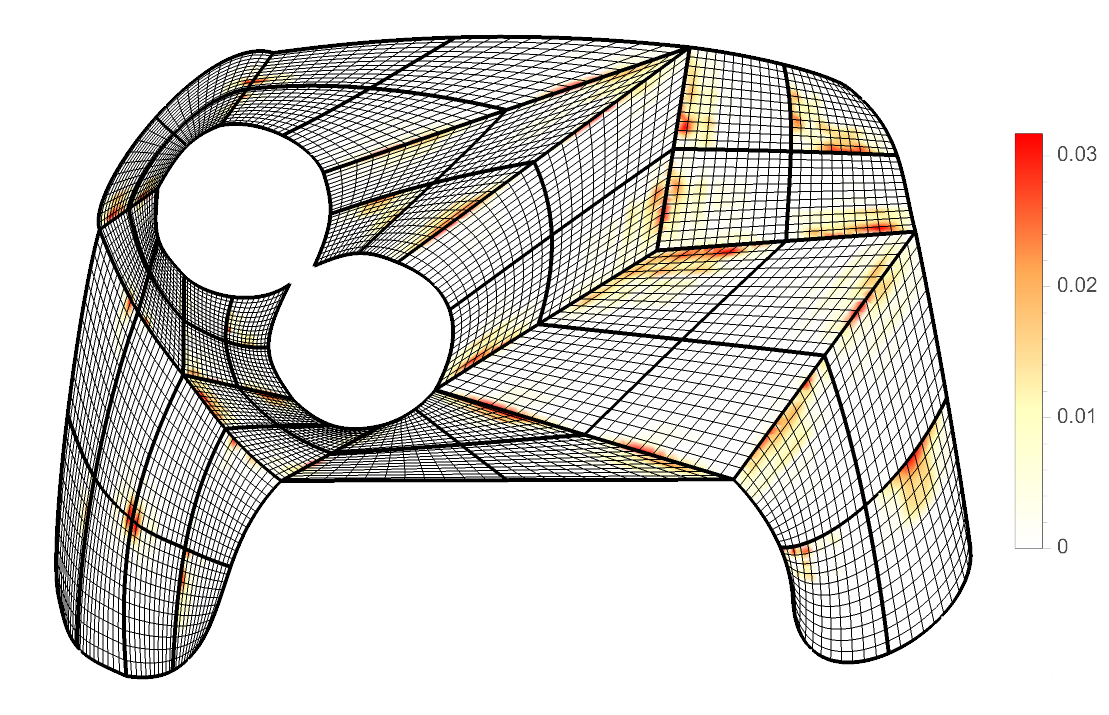}	
  \caption{{\new  Example~\ref{exNintendo}.  
  Error plots for the local method (left) and the global method (right) for the 52-patch domain.}}
  \label{nintendoCompare2}
\end{figure}
\end{ex}

\paragraph{AS-$G^1$ multi-patch surfaces}

We will present two examples of closed surfaces followed by an example of an open surface. 

\begin{ex}\label{exGoursat}
As first example of a closed surface, we consider the six-patch surface~$\f{S}$ shown in Fig.~\ref{goursatCompare}, left, which represents the \textit{Goursat} surface given by the implicit equation
$$
x^4 + y^4 + z^4 - 96(x^2+y^2+z^2) = 48,
$$
and which is generated by using the bubble-patch approach~\cite{KaByJu11} following the construction concept presented in~\cite[Example~5]{KaSaTa17b}. We approximate the given $G^1$-smooth but non-AS-$G^1$ multi-patch surface~$\f{S}$ by an AS-$G^1$ multi-patch spline surface~$\f{F}$, see Fig.~\ref{goursatCompare}, middle, whose surface patch parameterizations~$\f{F}^{(i)}$, $i \in \mathcal{I}_{\Omega}$ belong to the spline space~$(\mathcal{S}^{\f{4},\f{1}}_{\f{3}})^3$. The given surface~$\f{S}$ and the constructed AS-$G^1$ spline surface~$\f{F}$ exhibit a great similarity when {\new examining the differences between them in Fig.~\ref{goursatCompare}, right}. 
The similarity is further confirmed with the relative $L^2$- and $H^1$-error being equal to $\epsilon_{L^2} = 2.5 \cdot 10^{-4}$ and $\epsilon_{H^1} = 9.7 \cdot 10^{-4}$, respectively. 

\begin{figure}[ht]
	\centering
        \includegraphics[scale=0.575]{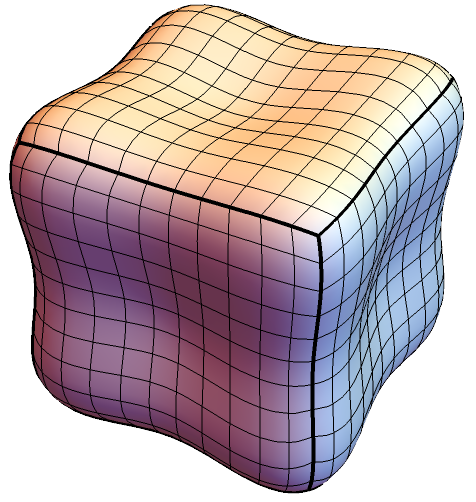}
        \hspace{0.3cm}
		\includegraphics[scale=0.58]{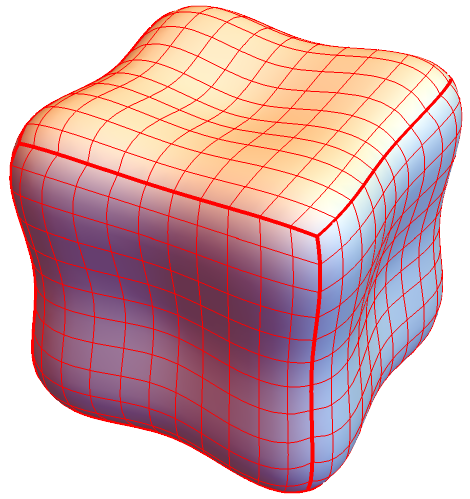}
  \hspace{0.3cm}
  \includegraphics[scale=0.405]{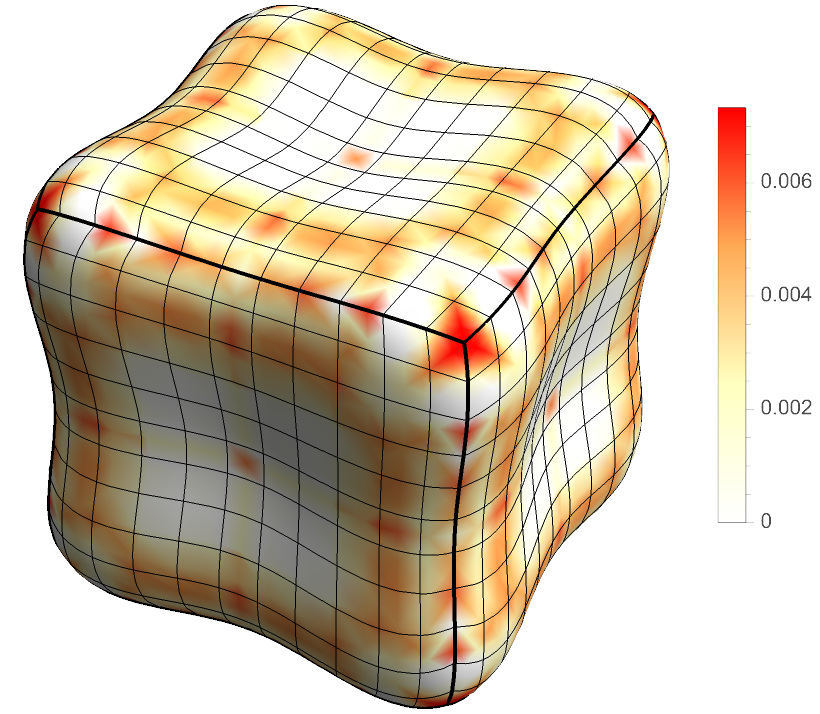}
	\caption{Example~\ref{exGoursat}. The given non-AS-$G^1$ multi-patch surface~$\f{S}$ (left), its AS-$G^1$ approximant~$\f{F}$ (middle) and {\new the differences }
 between them (right).}	
 \label{goursatCompare}
\end{figure}
\end{ex}

\begin{ex}\label{exTorus}
The second closed surface example deals with the double torus shown in Fig.~\ref{torusCompare}, above. This given $G^1$-smooth but not AS-$G^1$ multi-patch surface~$\f{S}$ consists of $74$ surface patches, and is again constructed via the bubble-patch approach~\cite{KaByJu11} representing now the implicitly given double torus
$$
\left(x(x-1)^2(x-2)+y^2\right)^2 + z^2 = \frac{1}{100}.
$$
We approximate the surface~$\f{S}$ by an AS-$G^1$ multi-patch spline surface~$\f{F}$ (see Fig.~\ref{torusCompare}, below) using the spline space $(\mathcal{S}^{\f{5},\f{1}}_{\f{2}})^3$ for the single surface patch parameterizations~$\f{F}^{(i)}$, $i \in \mathcal{I}_{\Omega}$. Again, a good approximation of the given non-AS-$G^1$ surface~$\f{S}$ by its AS-$G^1$ approximant~$\f{F}$ can be observed, cf. Fig.~\ref{torusCompare}, with a relative $L^2$- and $H^1$-error of $\epsilon_{L^2} =2.1 \cdot 10^{-3}$ and $\epsilon_{H^1} = 4.5 \cdot 10^{-3}$, respectively. 

\begin{figure}[ht]
	\centering
        \includegraphics[scale=0.45]{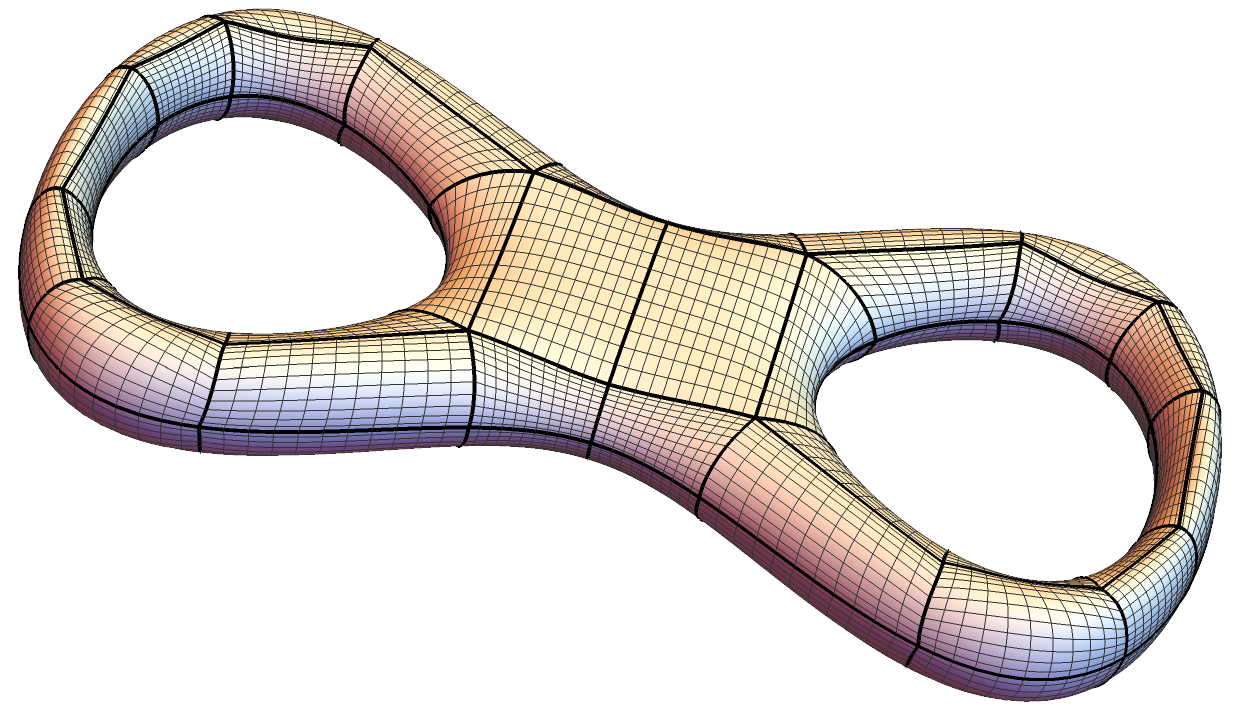}
        \vskip-2.5em
		\includegraphics[scale=0.45]{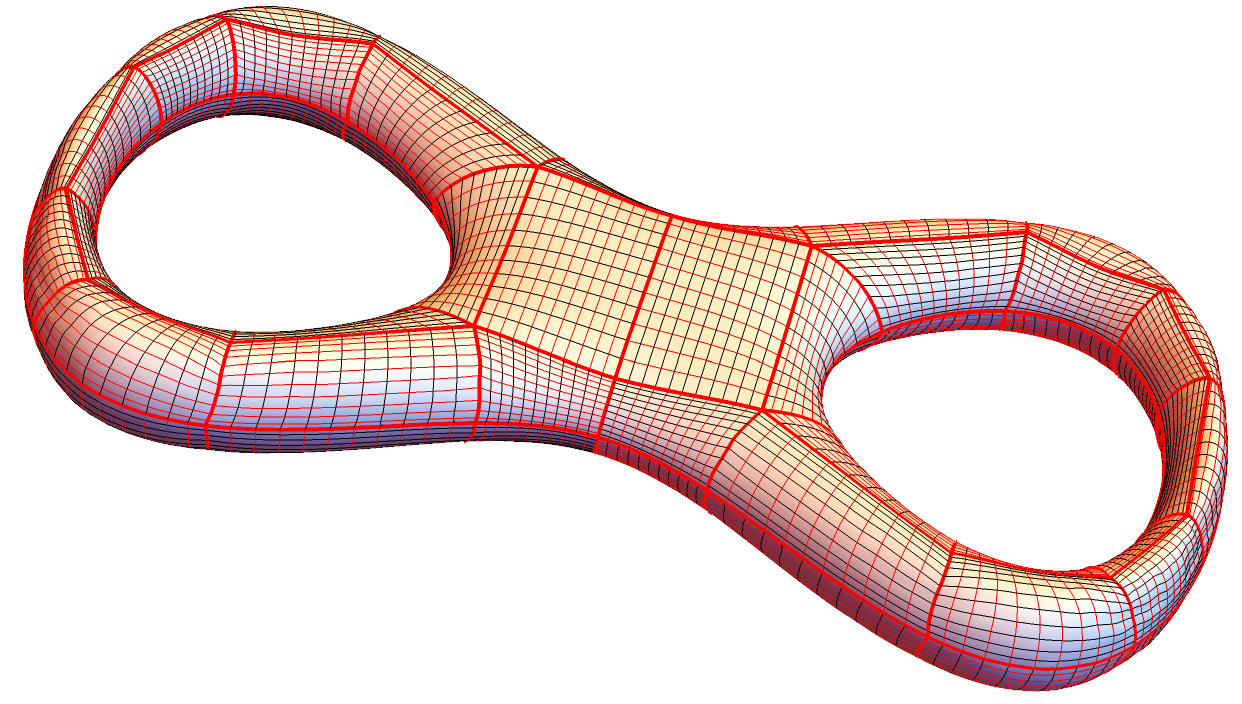}
		\caption{Example \ref{exTorus}. The given non-AS-$G^1$ multi-patch surface~$\f{S}$ (above) and its AS-$G^1$ approximant~$\f{F}$ (below). }
  \label{torusCompare}
\end{figure}
\end{ex}

\begin{ex}\label{exGspline}

In the last example, we consider an open multi-patch surface, namely the G-spline surface {\new $\f{S}$} presented in Fig.~\ref{GSplineCompare}, left, which consists of $36$ biquadratic patches and which possesses two extraordinary vertices. G-spline surfaces~\cite{Re95} are a simple and powerful tool to model complex multi-patch spline surfaces with extraordinary vertices. However, they are not AS-$G^1$ and have therefore a reduced approximation power as currently shown in~\cite{WeFaLiWeCa23}. 

To use our local based algorithm to approximate the G-spline surface {\new $\f{S}$} by an AS-$G^1$ multi-patch spline surface {\new $\f{F}$}, we will have to do a precomputation step for the given G-spline surface {\new $\f{S}$}, which will be demonstrated on the basis of the upper left extraordinary vertex, cf. Fig.~\ref{GSplineCompareSmall}. As already mentioned in Section~\ref{subsec:precomputation_data}, the $C^1$-$G^1$ transition {\new of the G-spline surface~$\f{S}$} between the first and second ring of patches around an extraordinary vertex (i.e. the blue interfaces in Fig.~\ref{GSplineCompareSmall}, left) 
{\new 
would not allow directly the selection of linear inner ring patch gluing functions~$\alpha^{(i,i_j)}$, since they would have to be at least of degree~$2$ in this case.} 
{\new Namely, by} enforcing these inner ring patch gluing functions~$\alpha^{(i,i_j)}$ to be linear as {\new required for $\f{F}$} in our proposed local method would imply singularities at the red vertices in Fig.~\ref{GSplineCompareSmall}, left. One way to prevent now the appearance of these singularities and to approximate a $G$-spline surface {\new $\f{S}$} by an AS-$G^1$ multi-patch spline surface {\new $\f{F}$} in a good quality, {\new and hence to get linear gluing functions $\alpha^{(i,i_j)}$ for $\f{F}$,} is to first irregularly subdivide some particular patches in the second ring around an extraordinary vertex with five new patches of degree~$4$ as shown in Fig.~\ref{GSplineCompareSmall}, middle. {\new Note that a regular subdivision would not work here, since the problem for the corresponding gluing functions would be just transmitted from the original surface $\f{S}$ to the subdivided surface.
}

In case of the given G-spline surface with $36$ patches (Fig.~\ref{GSplineCompare}, left), we subdivide now $9$ patches which results in a G-spline surface~$\f{S}$ consisting of $72$ patches (Fig.~\ref{GSplineCompare}, middle). The subdivided non-AS-$G^1$ surface~$\f{S}$ is approximated by an AS-$G^1$ multi-patch spline surface~$\f{F}$ by using splines from the space $(\mathcal{S}^{\f{4},\f{1}}_{\f{2}})^3$ to represent the single surface patch parameterizations~$\f{F}^{(i)}$, $i \in \mathcal{I}_{\Omega}$, and is shown in  Fig.~\ref{GSplineCompare}, right. The relative $L^2$- and $H^1$- error between the surfaces~$\f{S}$ and $\f{F}$ is equal to $\epsilon_{L^2} = 2.3 \cdot 10^{-4}$ and $\epsilon_{H^1} = 8.9 \cdot 10^{-4}$, respectively.

\begin{figure}[ht]
	\centering
        \includegraphics[scale=0.35]{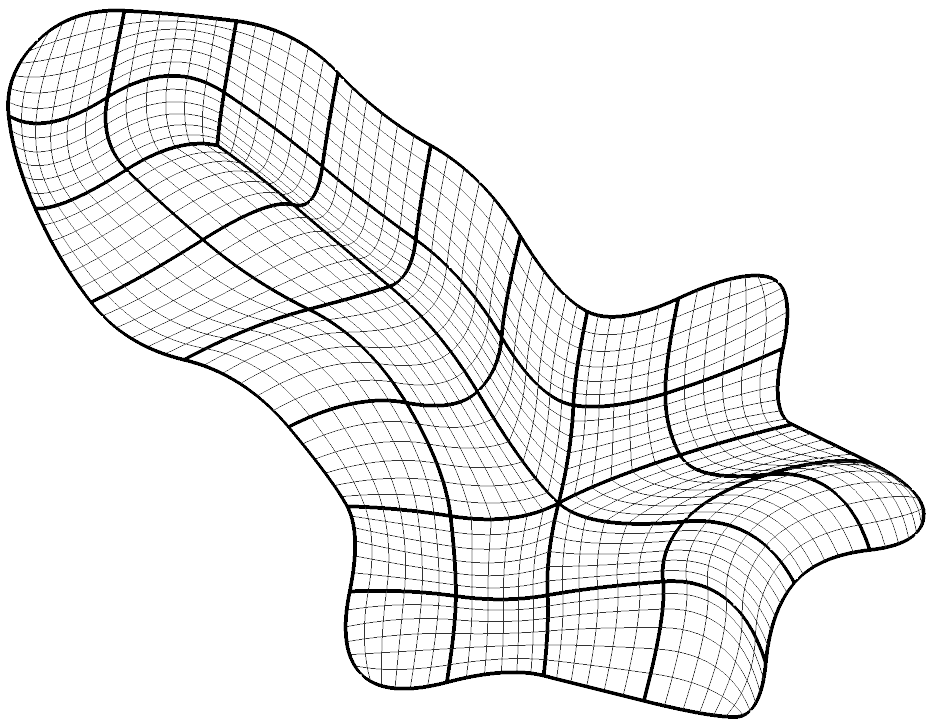}
        \hspace{-0.75cm}
        \includegraphics[scale=0.35]{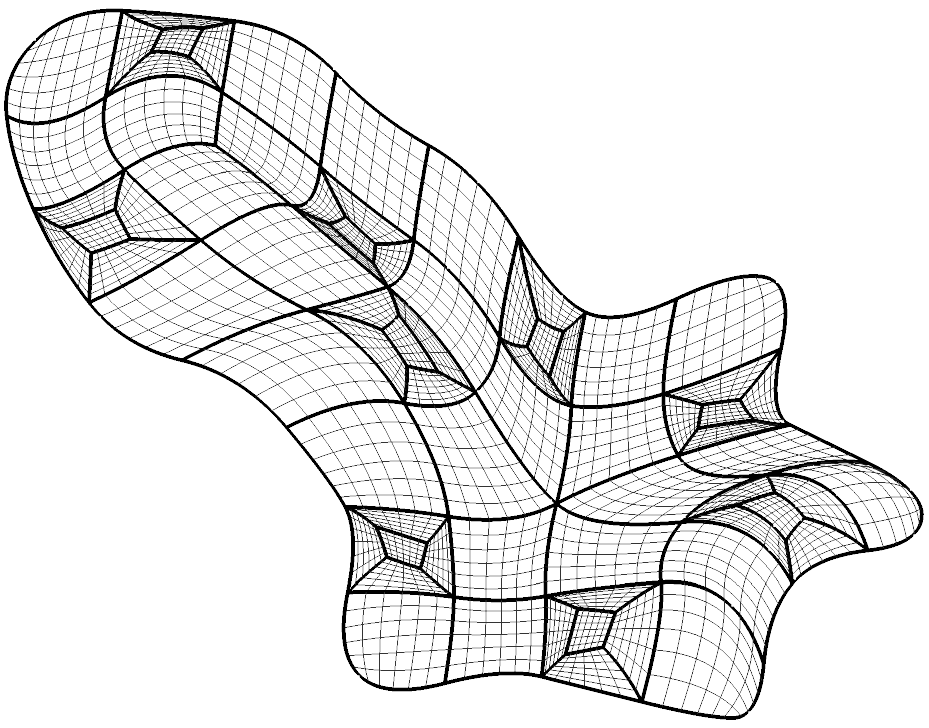}
        \hspace{-0.85cm}
		\includegraphics[scale=0.29]{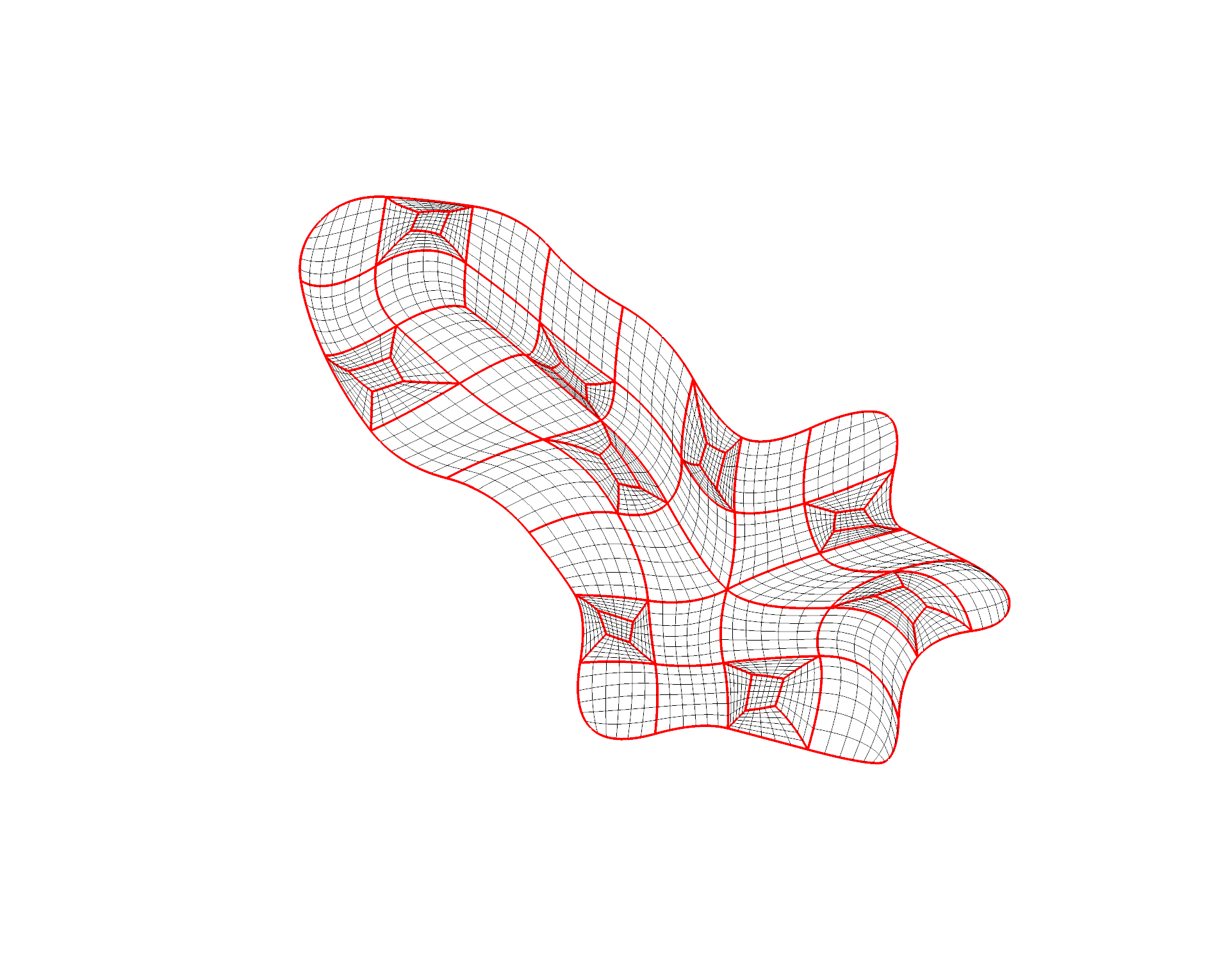}
		\caption{Example \ref{exGspline}. The originally given and the subdivided multi-patch surface~$\f{S}$ (left and middle) with the AS-$G^1$ approximant~$\f{F}$ (right). }
  \label{GSplineCompare}
\end{figure}

\begin{figure}[ht]
	\centering
         \includegraphics[scale=0.27]{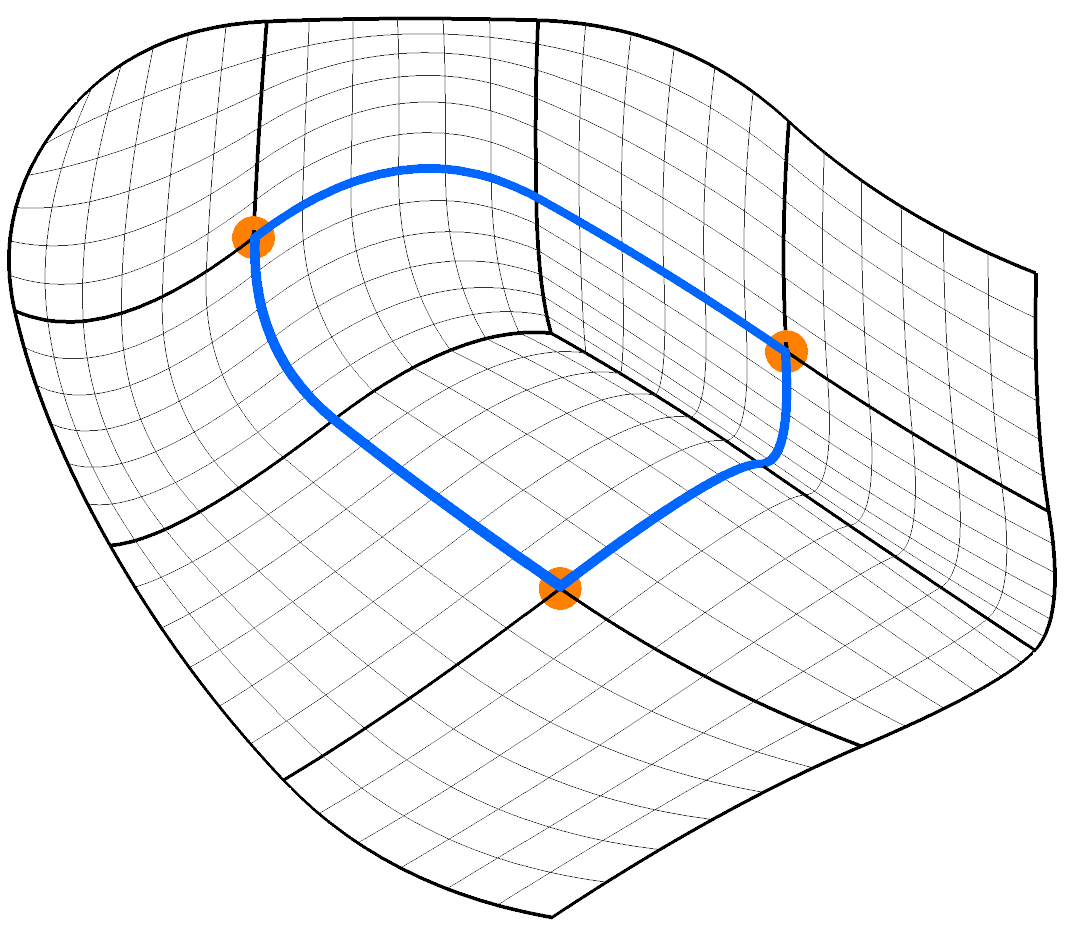}
         \hspace{0.3cm}
        \includegraphics[scale=0.34]{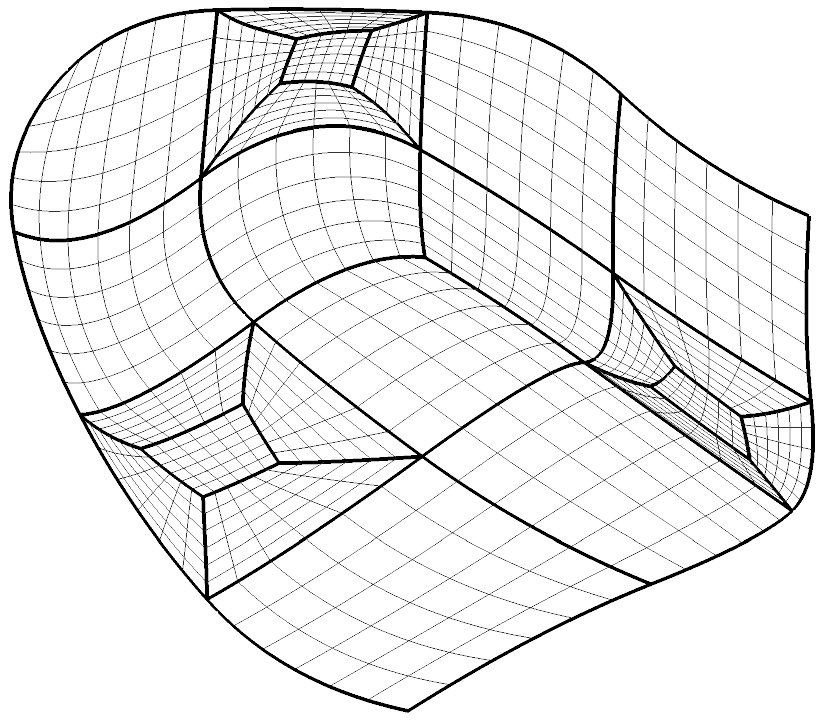}
         \hspace{0.3cm}
		\includegraphics[scale=0.38]{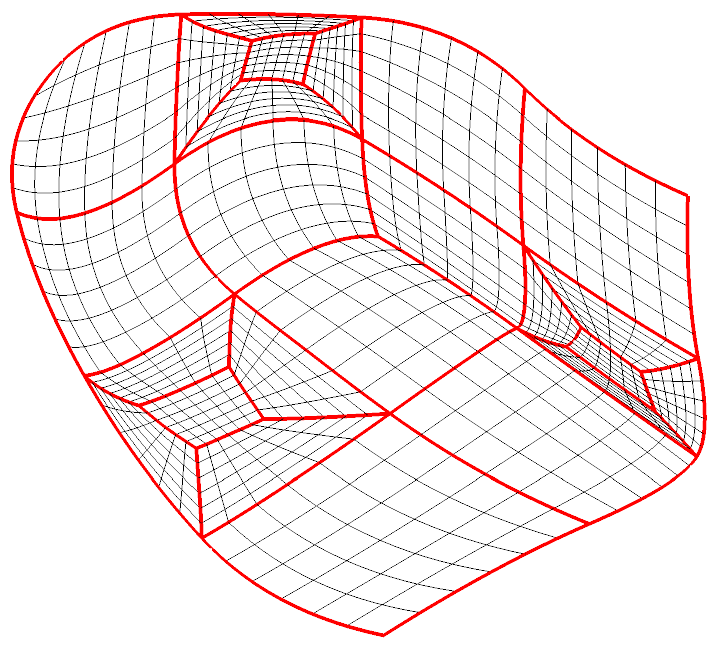}
		\caption{Example \ref{exGspline}. The upper left piece of the originally given and of the subdivided multi-patch surface~$\f{S}$ (left and middle) and of the AS-$G^1$ approximant (right) from Fig.~\ref{GSplineCompare}. In the left figure the blue interfaces separate the first and the second ring of patches around the extraordinary vertex. Red vertices represent the vertices where irregularities would appear if no subdivision would be applied on the three particular patches in the second ring.}
  \label{GSplineCompareSmall}
\end{figure}
\end{ex}

\subsection{Solving the biharmonic problem} \label{subsec:biharmonic_equation}

We will show that the use of an AS-$G^1$ multi-patch surface together with the $C^1$-smooth isogeometric spline space~\cite{FaJuKaTa22} allows the solving of fourth order PDEs, such as the biharmonic problem in this work, with optimal rates of convergence. This will be demonstrated below on the basis of Example~\ref{nintendoBIH} and \ref{goursatBIH} by employing the AS-$G^1$ multi-patch surfaces constructed in Example~\ref{exNintendo} and \ref{exGoursat}, respectively. 

In case of Example~\ref{nintendoBIH}, where the resulting planar multi-patch domain~$\Omega$ from Example~\ref{exNintendo} will be considered, we will study the following problem. Find $u: \Omega \rightarrow \mathbb{R}$, such that
\begin{equation}\label{biharmonicOpen}
\left\{
\begin{matrix}
\Delta^2_\Omega u(\bfm{x}) = f(\bfm{x}),  & \hspace{-0.5em} \bfm{x} \in \Omega, \\
\quad \; \; u(\bfm{x}) = g_1(\bfm{x}), & \bfm{x} \in \partial\Omega, \\
\; \; \partial_{\bfm{n}}u(\bfm{x}) = g_2(\bfm{x}), & \bfm{x} \in \partial\Omega,
\end{matrix}
\right.
\end{equation}
for some given functions $f$, $g_1$ and $g_2$. In Example~\ref{goursatBIH}, where we will deal with the closed multi-patch surface domain~$\Omega$ from Example~\ref{exGoursat}, which does not possess a boundary, we will consider instead a slightly different fourth order problem. Find $u: \Omega \rightarrow \mathbb{R}$ such that
\begin{equation}\label{biharmonicClosed}
\Delta_{\Omega}^2u(\bfm{x}) + \lambda u(\bfm{x}) = f(\bfm{x}), \quad \bfm{x}\in\Omega,
\end{equation}
for some given function~$f$, with $\lambda > 0$, where the additional term $\lambda u$ is needed to ensure well-posedness of the problem. Both problems~\eqref{biharmonicOpen} and \eqref{biharmonicClosed} will be solved via their weak form and a standard Galerkin discretization, see for details e.g.~\cite{FaJuKaTa22}, by finding the solution
\[
u_h(\bfm{x}) = \sum_{j\in\mathcal{J}}c_j \phi_{j,h}(\bfm{x}),\quad \bfm{x} \in \Omega,
\]
where $\{\phi_{j,h}\}_{j\in \mathcal{J}}$ is the basis of $C^1$-smooth isogeometric spline space $\mathcal{A}_h$, introduced in~\cite{FaJuKaTa22} and presented in Section~\ref{subsec:C1_spaces}, with the mesh size~{\new $h = \frac{1}{k+1}$} and $\mathcal{J} = \{0,1,\dots,\dim\mathcal{A}_h-1 \}$. {\new Note that the spline space~$\mathcal{A}$ is h-refinable as a direct consequence of its definition~\eqref{eq:definition} due to the h-refinability of the spline spaces $\mathcal{S}_{k}^{p,r+1}$ and $\mathcal{S}_{k}^{p-1,r}$ for the traces and specific transversal derivatives of the functions along/across the interface curves, see also~\cite{BrGiKaVa23}.} 

To {\new numerically} study the approximation properties of the $C^1$-smooth isogeometric spline space {\new $\mathcal{A}_h$} over the constructed AS-$G^1$ multi-patch surfaces, we will perform $h$-refinement by generating a sequence of $C^1$-smooth isogeometric spline spaces $\mathcal{A}_h$, with mesh sizes $h = 2^{-L}h_0$, where $L=0,1,2,\dots$ is the refinement level and $h_0$ is the initial mesh size. In Example~\ref{nintendoBIH} we will do the error analysis with respect to an exact solution~$u_{ex}$ by computing the errors with respect to the $L^2$-norm, to the $H^1$-seminorm and to the $L^2$-norm of the difference of the Laplace operator
\begin{equation*}
    {||\Delta u_{ex}-\Delta u_h||_{L^2}},
\end{equation*}
which is equivalent to the $H^2$-seminorm cf.~\cite{BaDe15}. When one global parameterization of the exact solution is not available as in Example~\ref{goursatBIH}, we will compare instead the difference between two subsequent solutions and will compute for this the corresponding analogue $h - \frac{h}{2}$ type error estimators given by
\begin{equation*}
        ||u_h - u_{\frac{h}{2}}||_{L^2}, \mbox{ }
        |u_h - u_{\frac{h}{2}}|_{H^1} \mbox{ and } 
        ||\Delta u_h - \Delta u_{\frac{h}{2}}||_{L^2}.
\end{equation*}
For the sake of brevity we will denote the different errors for Example~\ref{nintendoBIH} and \ref{goursatBIH} simply as errors with respect to the $L^2$-norm, and $H^1$- and $H^2$-seminorms.

\begin{ex}\label{nintendoBIH}
We solve the biharmonic problem~\eqref{biharmonicOpen} over the planar AS-$G^1$ multi-patch geometry constructed in Example~\ref{exNintendo} {\new visualized in Fig.~\ref{nintendoCompare} (middle)}. For this purpose, we choose the exact solution 
\begin{equation} \label{exactSolution}
u_{ex}(x_1,x_2) = \cos(4x_1)\sin(4x_2),
\end{equation}
compute the functions $f$, $g_1$ and $g_2$ from it, {\new strongly impose via $L^2$ projection the boundary data~$g_1$ and $g_2$ to the numerical solution~$u_h$, which will be represented by means of }
the $C^1$-smooth isogeometric spline space $\mathcal{A}_h$ {\new with} 
degree $p=4$, 
regularity $r=1$ and 
{\new an} initial mesh size $h_0=\frac{1}{6}$. The numerical solution at the highest refinement level as well as the convergence rates for the $L^2$-norm and $H^1$- and $H^2$-seminorms, where all rates are of optimal order, are presented in Fig.~\ref{nintendoBiharmonic}.
\begin{figure}[ht]
	\centering
        \includegraphics[scale=0.35]{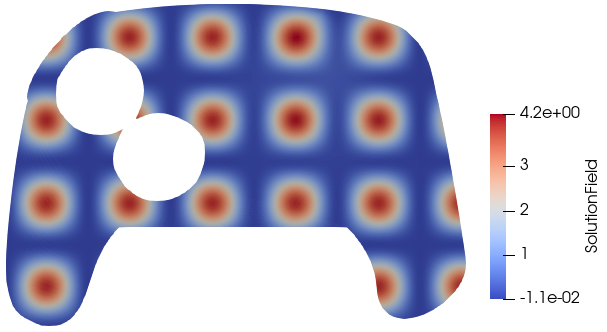}
        \hskip0.5em
		\includegraphics[scale=0.21]{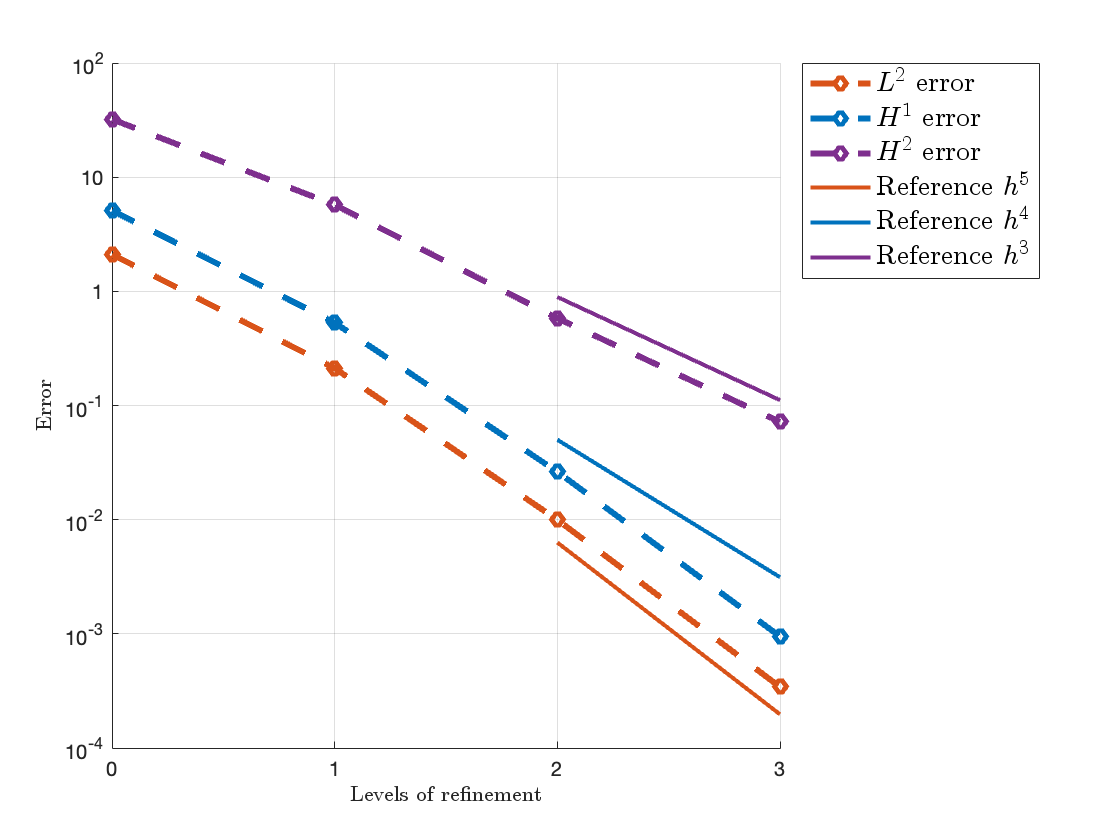}
      \caption{Example~\ref{nintendoBIH}. Solving the biharmonic problem~\eqref{biharmonicOpen} over the constructed AS-$G^1$ multi-patch geometry from Example~\ref{exNintendo} {\new given in Fig.~\ref{nintendoCompare} (middle)} for the exact solution~\eqref{exactSolution} by using the spline space~$\mathcal{A}_h$ for $p=4$ and $r=1$. Numerical solution at the highest refinement level (left) and the resulting convergence plots (right). }
  \label{nintendoBiharmonic}
\end{figure}
\end{ex}

\begin{ex}\label{goursatBIH}
We solve the biharmonic problem~\eqref{biharmonicClosed} over the constructed AS-$G^1$ multi-patch approximation of the Goursat surface from Example \ref{exGoursat} {\new shown in Fig.~\ref{goursatCompare} (middle)}. We choose
\begin{equation} \label{eq:function_f}
f(x_1,x_2,x_3) = \cos\left(\frac{x_1}{2}\right)\cos\left(\frac{x_2}{2}\right)\cos\left(\frac{x_3}{2}\right),
\end{equation}
and select as degree $p=4$, as regularity $r=1$ and as initial mesh size $h_0 = \frac{1}{2}$ for the $C^1$-smooth isogeometric spline space $\mathcal{A}_h$. Fig.~\ref{goursatBiharmonic} presents the numerical solution at the highest refinement level as well as the convergence rates for the $L^2$-norm, and $H^1$- and $H^2$-seminorms. Again all convergence rates are of optimal order. 
\begin{figure}[ht]
	\centering
        \includegraphics[scale=0.25]{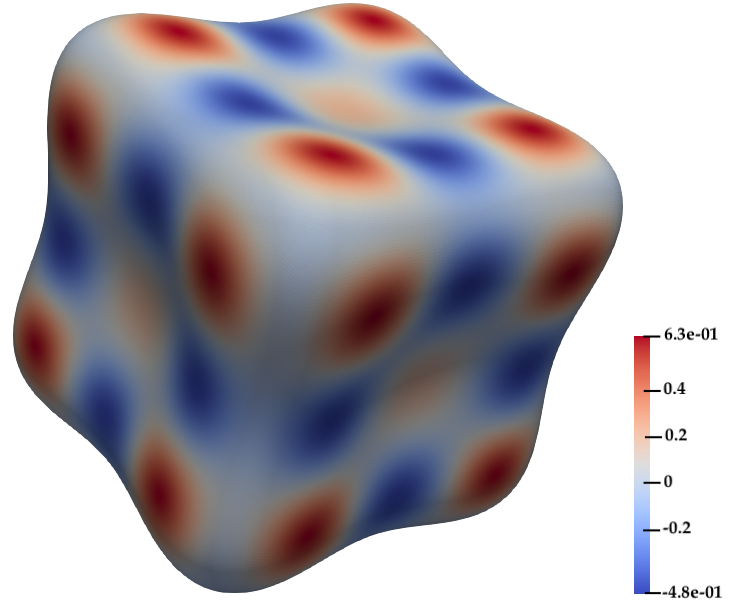}
        \hskip2em
		\includegraphics[scale=0.21]{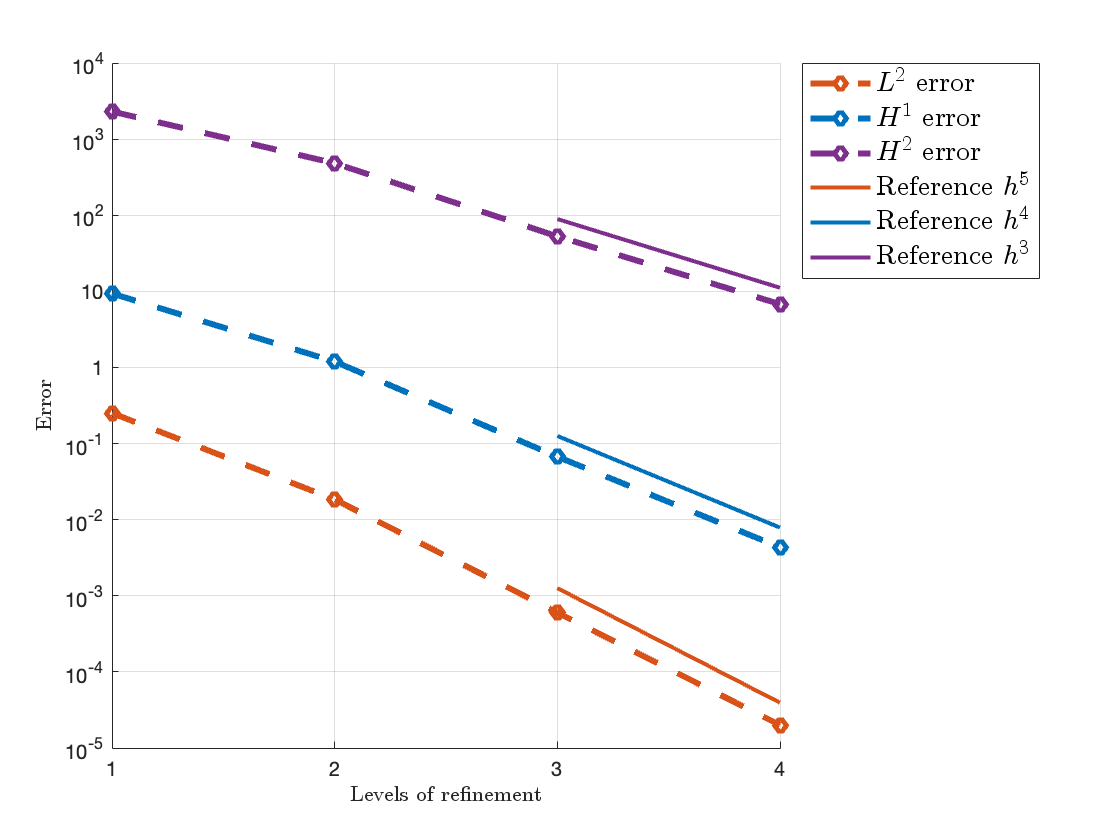}
       \caption{Example~\ref{goursatBIH}. Solving the biharmonic problem~\eqref{biharmonicClosed} over the AS-$G^1$ approximation of the Goursat surface as presented in Example~\ref{exGoursat} {\new given in Fig.~\ref{goursatCompare} (middle)} for the function~$f$ given in~\eqref{eq:function_f} by using the spline space~$\mathcal{A}_h$ for $p=4$ and $r=1$. Numerical solution at the highest refinement level (left) and the resulting convergence plots (right). }
  \label{goursatBiharmonic}
\end{figure}
\end{ex}
\section{Conclusion} \label{sec:conclusion}

We presented a novel method for the approximation of a given $G^1$-smooth but non-AS-$G^1$ multi-patch surface by an AS-$G^1$ multi-patch spline surface. The proposed technique is fully local by numerically solving in three consecutive steps, namely first for all inner vertices and boundary vertices of patch valency greater than one, then for all interface {\new curves} and finally for all patches, a series of small quadratic optimization problems. This makes the approach also applicable to $G^1$-smooth multi-patch surfaces with a high number of surface patches as shown on the basis of several examples. The potential of our approach was further demonstrated by solving with optimal rates of convergence the biharmonic problem over some of the constructed AS-$G^1$ multi-patch spline surfaces by means of the $C^1$-smooth isogeometric spline space~\cite{FaJuKaTa22}. One possible topic for future research could be the use of our technique to generate complex AS-$G^1$ multi-patch spline surfaces to represent the computational domain of other fourth order partial differential equations such as the Kirchhoff-Love shell problem, the Cahn-Hilliard equation or problems of strain gradient elasticity, and to solve the corresponding problems over these surfaces. 

\section*{Acknowledgements} 
{\new The authors wish to thank the anonymous reviewers for their comments that helped to improve the paper.} A. Farahat and M. Kapl have been partially supported by the Austrian Science Fund (FWF) through the project P~33023-N.
V.~Vitrih has been partially supported by the Slovenian Research {\new and Innovation} Agency (research program P1-0404 and research projects N1-0296, J1-1715, N1-0210 and J1-4414). A.~Kosma\v c has been partially supported by the Slovenian Research {\new and Innovation} Agency (research program P1-0404, research project N1-0296 and Young Researchers Grant).
This support is gratefully acknowledged.


\end{document}